\documentclass{article}

\usepackage{showlabels}

\usepackage{combelow}
\usepackage{caption}
\usepackage{subcaption} 
\usepackage{amsthm,amsfonts,amssymb,amsmath,amscd} 
\usepackage{aliascnt} 
\usepackage{mathrsfs} 
\usepackage{xargs,ifthen} 
\usepackage{multirow,array} 

\usepackage{nomencl} 
\makenomenclature

\usepackage{comment}
\usepackage{color}
\usepackage{xypic} 
\usepackage{tikz}
\usetikzlibrary{calc}

\usepackage{hyperref}
\definecolor{sxdarkblue}{RGB}{150,180,200} 
\definecolor{sxlightblue}{RGB}{186,215,230} 
\definecolor{sxred}{RGB}{153,0,0} 
\definecolor{tocolor}{rgb}{.1,.1,.1}
\definecolor{urlcolor}{rgb}{.2,.2,.6}
\definecolor{linkcolor}{rgb}{.1,.1,.5}
\definecolor{citecolor}{rgb}{.4,.2,.1}
\hypersetup{backref=true, colorlinks=true, urlcolor=urlcolor, linkcolor=linkcolor, citecolor=citecolor}

\newcommandx{\thdef}[2]{
	\newaliascnt{#1}{theorem}  
	\newtheorem{#1}[#1]{#2}
	\aliascntresetthe{#1}  
	\newtheorem*{#1*}{#2}
	\expandafter\newcommand\expandafter{\csname #1autorefname\endcsname}{#2}
}

\makeatletter
\newtheorem*{rep@theorem}{\rep@title}
\newcommand{\newreptheorem}[2]{%
\newenvironment{rep#1}[1]{%
 \def\rep@title{#2 \ref{##1}}%
 \begin{rep@theorem}}%
 {\end{rep@theorem}}}
\makeatother

\newtheorem{theorem}{Theorem}[section]
\newreptheorem{theorem}{Theorem}

\thdef{lemma}{Lemma}
\thdef{corollary}{Corollary}
\thdef{conjecture}{Conjecture}
\thdef{question}{Question}
\thdef{problem}{Problem}
\newreptheorem{conjecture}{Conjecture}
\thdef{proposition}{Proposition}

\theoremstyle{definition}
\thdef{definition}{Definition}
\thdef{notation}{Notation}

\theoremstyle{remark}
\thdef{remark}{Remark}

\theoremstyle{remark}
\thdef{example}{Example}

\newcommand{\defn}[1]{\textbf{\textit{#1}}} 

\newcommand{\spc}[1]{\mathsf{#1}} 
\newcommand{\shf}[1]{\mathcal{#1}} 

\newcommand{\CC}{\mathbb{C}}
\newcommand{\ZZ}{\mathbb{Z}}
\newcommand{\NN}{\mathbb{N}}

\newcommand{\PP}[1][]{\mathbb{P}^{#1}}

\newcommand{\Gr}[1]{\fn{\spc{Gr}}[{#1}]}

\newcommand{\rbrac}[1]{\left(#1\right)} 
\newcommand{\abrac}[1]{\left\langle#1\right\rangle} 

\newcommandx{\fn}[2][2=]{#1\ifthenelse{\equal{#2}{}}{}{\!\rbrac{{#2}}}} 
\newcommandx{\id}[2][2=]{\fn{{\rm id}_{#1}}[#2]} 

\newcommand{\ext}[2][\bullet]{\spc{\Lambda}^{#1}{#2}} 
\newcommandx{\sym}[3][1=\bullet,3=]{\spc{S}^{#1}_{#3}{#2}} 
\newcommand{\tens}[2][\bullet]{\spc{T}^{#1}{#2}} 
\newcommandx{\End}[2][1=]{\fn{\spc{End}_{#1}}[#2]} 
\newcommandx{\Hom}[2][1=]{\fn{\spc{Hom}_{#1}}[#2]} 
\newcommandx{\Aut}[2][1=]{\fn{\spc{Aut}_{#1}}[#2]} 
\newcommandx{\aut}[2][1=]{\fn{\mathfrak{aut}_{#1}}[{#2}]} 
\newcommandx{\image}[1]{\fn{\spc{img}}[#1]} 
\renewcommandx{\ker}[1]{\fn{\spc{ker}}[#1]} 
\newcommandx{\rank}[1]{\fn{\mathrm{rank}}[#1]} 

\newcommandx{\ann}[1]{\fn{\spc{ann}}[\spc{#1}]} 

\newcommandx{\hlgy}[3][1=\bullet,3=]{\spc{H}_{#1}^{#3}\!\rbrac{{#2}}} 
\newcommandx{\cohlgy}[3][1=\bullet,3=]{\spc{H}^{#1}_{#3}\!\rbrac{{#2}}} 
\newcommandx{\hypcohlgy}[3][1=\bullet,3=]{\mathbb{H}^{#1}_{#3}\!\rbrac{{#2}}} 
\newcommandx{\chow}[3][1=\bullet,3=]{\spc{A}^{#1}_{#3}\!\rbrac{{#2}}} 

\newcommandx{\Ext}[3][1=\bullet,3=]{\fn{\spc{Ext}^{#1}_{#3}}[{#2}]} 
\newcommandx{\Tor}[3][1=\bullet,3=]{\fn{\spc{Tor}^{#1}_{#3}}[{#2}]} 

\newcommandx{\Pic}[2][1=]{\fn{\spc{Pic}_{#1}}[{#2}]} 
\newcommandx{\chernalg}[2][1=\bullet]{\fn{\spc{Chern}^{#1}}[{#2}]} 
\newcommandx{\chern}[2][1=]{\fn{c_{#1}}[#2]} 
\newcommandx{\ch}[2][1=]{\fn{\mathrm{ch}_{#1}}[{#2}]} 

\newcommandx{\sKer}[2][1=]{ \fn{ \shf{K}er_{#1}}[{#2}] } 
\newcommandx{\sImg}[2][1=]{ \fn{ \shf{I}mg_{#1}}[{#2}] } 
\newcommandx{\sHom}[2][1=]{ \fn{ \shf{H}om_{#1}}[{#2}] } 
\newcommandx{\sEnd}[2][1=]{ \fn{ \shf{E}nd_{#1}}[{#2}] } 
\newcommandx{\sExt}[3][1=\bullet,3=]{\fn{\shf{E}xt^{#1}_{#3}}[{#2}]} 
\newcommandx{\sTor}[3][1=\bullet,3=]{\fn{\shf{T}or^{#1}_{#3}}[{#2}]} 
\newcommand{\sO}[1]{\shf{O}_{#1}}

\newcommandx{\ssym}[3][1=\bullet,3=]{\shf{S}ym^{#1}_{#3}{#2}} 

\newcommandx{\forms}[2][1=\bullet]{\Omega^{#1}_{#2}} 
\newcommandx{\can}[1][1=]{\omega_{#1}} 
\newcommandx{\acan}[1][1=]{\omega_{#1}^{-1}} 
\newcommandx{\tshf}[1]{\shf{T}_{#1}} 
\newcommandx{\mvect}[2][1=\bullet]{ \ext[#1]{\tshf{#2}} }
\newcommandx{\der}[2][1=\bullet]{\mathscr{X}^{#1}_{#2}} 
\newcommandx{\sJet}[3][1=,2=]{\shf{J}^{#1}_{#2}#3} 

\newcommandx{\tb}[2][1=]{\spc{T}_{#1}{#2}} 
\newcommandx{\ctb}[2][1=]{\spc{T}_{\!#1}^*{#2}} 
\newcommandx{\lie}[2][2=]{\fn{\mathscr{L}_{#1}}[#2]} 
\newcommandx{\hook}[2][2=]{\fn{i_{#1}}[#2]} 
\newcommand{\cvf}[1]{\partial_{{#1}}} 

\newcommandx{\Dgn}[2][1=]{\fn{\spc{Dgn}_{#1}}[#2]}
\newcommandx{\AD}[2][1=]{\fn{\spc{Dgn}_{#1}}[\ps_{#2}]}
\newcommandx{\Sing}[1]{\fn{\spc{Sing}}[#1]}
\newcommandx{\Sec}[2][1=]{\fn{\spc{Sec}_{#1}}[#2]}

\newcommand{\I}{\spc{I}}
\newcommand{\bL}{\spc{L}}

\newcommand{\M}{\spc{M}}
\newcommand{\bS}{\spc{S}}

\newcommand{\bR}{\spc{R}}

\newcommand{\W}{\spc{W}}
\newcommand{\X}{\spc{X}}

\newcommand{\Z}{\spc{Z}}

\newcommand{\V}{\spc{V}}

\newcommand{\D}{\spc{D}}

\newcommand{\sE}{\shf{E}}
\newcommand{\E}{\spc{E}}

\newcommand{\sL}{\shf{L}}

\newcommand{\A}{\spc{A}}

\newcommand{\GL}[1]{\fn{\spc{GL}}[#1]}
\newcommand{\SL}[1]{\fn{\spc{SL}}[#1]}

\newcommand{\g}{\mathfrak{g}}

\newcommand{\sln}[1]{\fn{\mathfrak{sl}}[#1]}

\newcommand{\G}{\spc{G}}

\newcommand{\ps}{\sigma} 
\newcommand{\QA}[1]{\fn{\spc{QA}}[{#1}]} 
\newcommand{\Proj}[1]{\fn{\spc{Proj}}[{#1}]} 

\newcommand{\CLimgwidth}{2in}

\renewcommand{\CLimgwidth}{1.2in}

\newcommand{\bpymemail}{pym@maths.ox.ac.uk}
\newcommand{\R}{\spc{R}}

\newcommand{\QDef}{\spc{QDef}}

\renewcommand{\QA}{\spc{QA}}

\newcommand{\PB}{\spc{PB}}

\newcommand{\CLNimage}{
\begin{figure}[t]
\begin{center}
\begin{tabular}{ccc}
$\bL(1,1,1,1)$ & $\bR(2,2)$ & $\bS(2,3)$  \\
\includegraphics[width=\CLimgwidth]{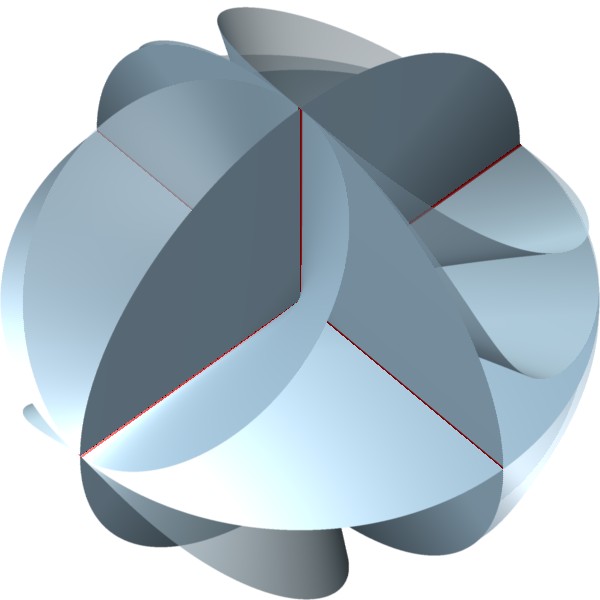}
 &
\includegraphics[width=\CLimgwidth]{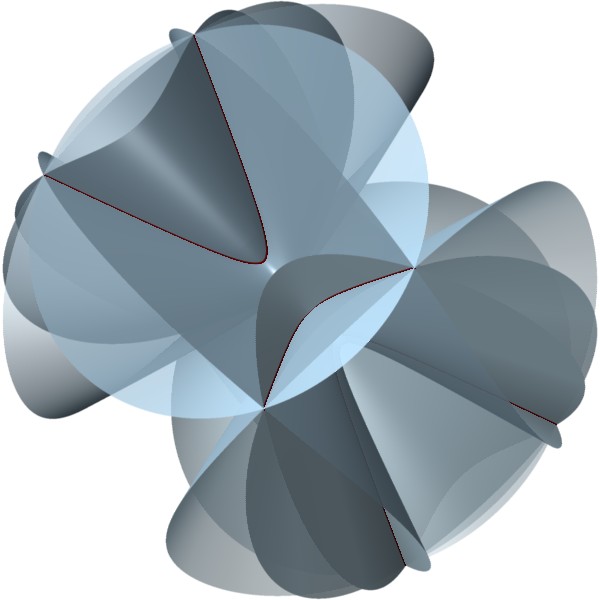}
&
\includegraphics[width=\CLimgwidth]{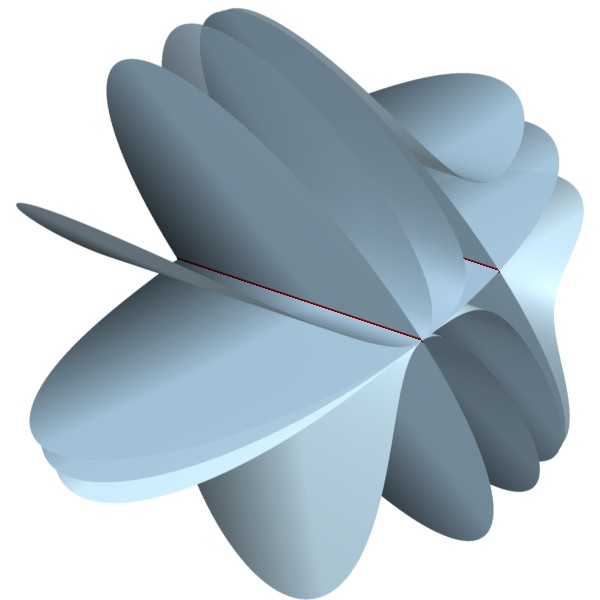}
\\
\includegraphics[width=\CLimgwidth]{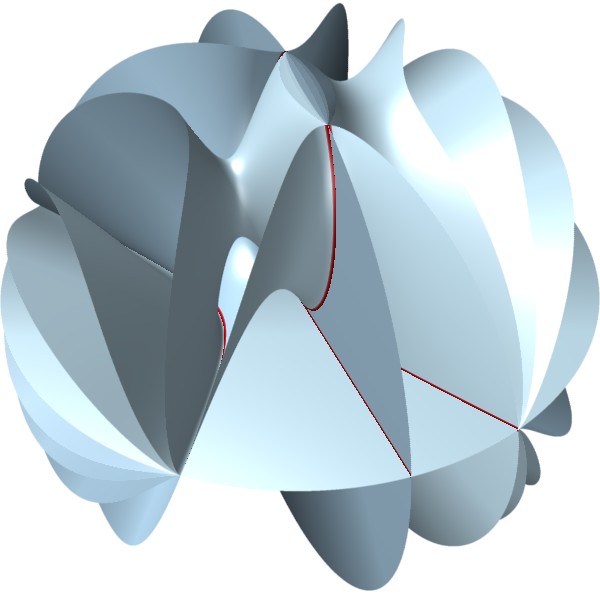}
&
\includegraphics[width=\CLimgwidth]{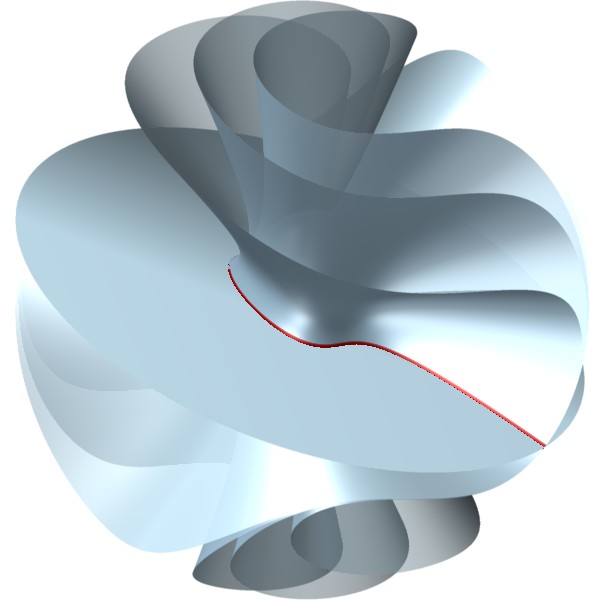} &
\includegraphics[width=\CLimgwidth]{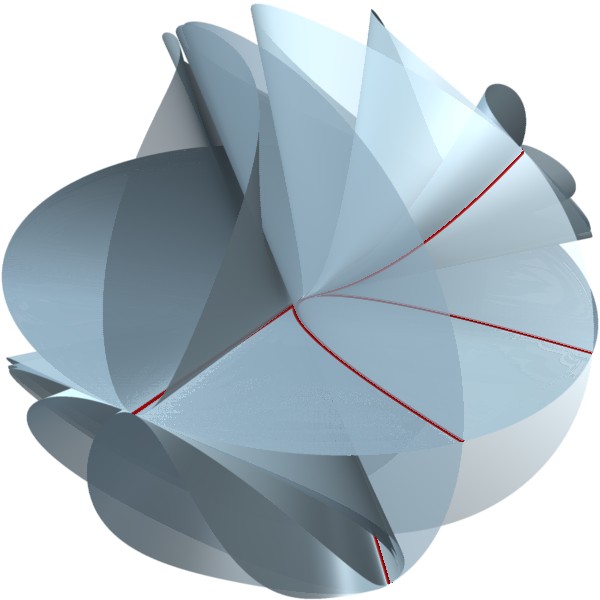}
\\
$\bL(1,1,2)$ & $\bR(1,3)$ & $\E(3)$\\ \ \\
\end{tabular}
\end{center}
\caption[The six families of Poisson structures on {$\PP[3]$}]{Real slices of Poisson structures on $\PP[3]$, one from each irreducible component in the classification.  The blue surfaces represent symplectic leaves, while the red curves represent one-dimensional components of the zero locus.  For ease of plotting, the $\bL(1,1,1,1)$, $\bL(1,1,2)$ and $\bS(2,3)$ examples shown here were chosen to have algebraic symplectic leaves.  They therefore do not represent ``generic'' elements of the corresponding families.  Nevertheless, they do exhibit several features of the generic geometry (such as the blue planes and red lines in the $\bL(1,1,1,1)$ case).}\label{fig:CLN}
\end{figure}
}

\begin{document}

\title{Quantum deformations of projective three-space}
\author{Brent Pym\footnote{Mathematical Institute, University of Oxford, \href{mailto:\bpymemail}{\texttt{\bpymemail}}}}
\maketitle

\begin{abstract}
We describe the possible noncommutative deformations of complex projective three-space by exhibiting the Calabi--Yau algebras that serve as their homogeneous coordinate rings.  We prove that the space parametrizing such deformations has exactly six irreducible components, and we give explicit presentations for the generic members of each family in terms of generators and relations.  The proof uses deformation quantization to reduce the problem to a similar classification of unimodular quadratic Poisson structures in four dimensions, which we extract from Cerveau and Lins Neto's classification of degree-two foliations on projective space.  Corresponding to the ``exceptional'' component in their  classification is a quantization of the third symmetric power of the projective line that supports bimodule quantizations of the classical Schwarzenberger bundles.
\end{abstract}

\tableofcontents

\section{Introduction}

Complex projective space $\PP[n]$, while it admits no deformations as a classical algebraic variety, nevertheless admits several interesting quantum (i.e.~noncommutative) deformations.  This paper is concerned with the classification of such deformations in the case $n=3$, and its relation to problems in Poisson geometry and foliation theory.

Following the philosophy of noncommutative projective geometry, we model quantum versions of $\PP[n]$ using their homogeneous coordinate rings, which are noncommutative analogues of polynomial rings and satisfy the Artin--Schelter (AS) regularity criterion.  This criterion essentially encodes the smoothness of the corresponding noncommutative $\CC^{n+1}$; it was introduced in \cite{Artin1987}, where the classification is achieved in the case when $n$ is at most two.

Shortly after their introduction, the AS regular algebras that correspond to noncommutative versions of the projective plane $\PP[2]$ and the quadric surface $\PP[1]\times \PP[1]$ were studied by Artin, Tate and Van den Bergh~\cite{Artin1990}, and Bondal and Polishchuk~\cite{Bondal1993a}, where the classification is restated in geometric terms.  Corresponding to the fact that a Poisson structure on a surface is determined by a section of the anticanonical bundle, there is a close link between the noncommutative algebras and the geometry of anticanonical divisors in the classical surfaces.

Since that work, a substantial theory of noncommutative curves and surfaces has been developed; see, for example, the survey article~\cite{Stafford2001}.  The study of higher-dimensional quantum projective varieties is an active area of research, and the case of $\PP[3]$, in the form of AS regular algebras of global dimension four, has attracted considerable attention.  In addition to the development in~\cite{Goetz2003} of several general properties of quantum $\PP[3]$s, numerous examples of their homogeneous coordinate rings are known; they include skew-polynomial rings, the four-dimensional Sklyanin algebras~\cite{Sklyanin1982,Smith1992}, extensions of three-dimensional Artin--Schelter regular algebras~\cite{Cassidy1999,LeBruyn1996}, double extension algebras~\cite{Zhang2009}; and the homogenization of the universal enveloping algebra of the $\sln{2}$ Lie algebra~\cite{LeBruyn1993}.  Nevertheless, the full classification remains ``one of the central questions in noncommutative projective geometry''~\cite{Lu2007}.

Amongst all of the AS regular algebras are those that are also Calabi--Yau in the sense of~\cite{Ginzburg2006}.  
This subclass is important because other AS regular algebras can be obtained from Calabi--Yau ones by various twisting procedures; see ~\cite{Goodman2013,He2013,Reyes2013}.  In fact, as we explain in \autoref{sec:cy}, although a given quantum deformation of $\PP[n]$ may be representable by many different AS regular algebras, there will be a unique one that is Calabi--Yau; this idea underlies Etingof and Ginzburg's approach~\cite{Etingof2010} to noncommutative del Pezzo surfaces.  Our main result can thus be viewed as a classification of the quantum deformations of $\PP[3]$:
\begin{theorem}\label{thm:qa-cmpts}
Flat deformations of the polynomial ring $\CC[x_0,x_1,x_2,x_3]$ as a graded Calabi--Yau algebra come in six irreducible families.  These families are realized as the closures of the $\GL{4,\CC}$-orbits of the normal forms summarized in \autoref{tab:p3quant}.
\end{theorem}

It is important to note that the families described in the theorem intersect nontrivially; it is not clear how to describe all of the possible intersections, which is why we focus in this paper on suitably generic elements of each family (a Zariski open set in each component).

The strategy of the proof of \autoref{thm:qa-cmpts} is to reduce the problem to the study of the semi-classical limits using Kontsevich's results~\cite{Kontsevich2001,Kontsevich2003} on deformation quantization.  Work of Dolgushev \cite{Dolgushev2009} shows that Kontsevich's deformation quantization of a given Poisson structure will be a Calabi--Yau algebra if and only if the Poisson structure is unimodular in the sense of Weinstein~\cite{Weinstein1997}, and so \autoref{thm:qa-cmpts} can be deduced from the related
\begin{theorem}\label{thm:pb-cmpts}
The variety parametrizing unimodular quadratic Poisson structures in four dimensions has exactly six irreducible components.
\end{theorem}
In \autoref{sec:cmpts}, we give explicit formulae for the generic Poisson brackets in each component. This classification of Poisson brackets is a consequence of the celebrated work of Cerveau and Lins Neto~\cite{Cerveau1996} regarding codimension-one foliations on $\PP[3]$ and the recent refinements of Loray, Pereira and Touzet~\cite{Loray2013}, together with the well-known correspondence between homogeneous quadratic Poisson structures and Poisson structures on projective space~\cite{Bondal1993,Polishchuk1997}, which we review.  The Poisson structures on $\PP[3]$ corresponding to the algebras in \autoref{tab:p3quant} are illustrated in \autoref{fig:CLN}.

\begin{table}[t]\setlength{\extrarowheight}{3pt}
\caption{Generic graded Calabi--Yau deformations of $\CC[x_0,x_1,x_2,x_3]$}\label{tab:p3quant}
\begin{tabular}{c|c|c|p{5.5cm}}
Type		& Orbit & Relations					&	Description \\
		& dimension & &\\
\hline \hline
$\bL(1,1,1,1)$	& 14		& \eqref{eqn:l1111-rels}	& A skew-polynomial ring that is Calabi--Yau~\cite[Example 5.5]{Reyes2013} \\
\hline
$\bL(1,1,2)$	& 17		& \eqref{eqn:l112-rels}		& An algebra from~\cite[Theorem 1.1]{Cassidy2006} that is Calabi--Yau \\ \hline
$\R(2,2)$		& 16		& \eqref{eqn:r22-rels}		& A four-dimensional Sklyanin algebra~\cite{Bocklandt2010,Sklyanin1982,Smith1992} \\ \hline
$\R(1,3)$		& 21		& \eqref{eqn:r13-rels}		& A central extension of a three-dimensional Sklyanin algebra~\cite{LeBruyn1996} \\ \hline
$\bS(2,3)$		& 17		& \eqref{eqn:s23-rels}		& An Ore extension of $\CC[x_0,x_1,x_2]$ by a divergence-free derivation \\ \hline
$\E(3)$			& 13			& \eqref{eqn:e3-rels}		& The quantization of the third symmetric power of $\PP[1]$
\end{tabular}
\end{table}

\CLNimage

As indicated, most of the algebras in \autoref{tab:p3quant} are familiar from the literature.  Remarkably, in order to obtain the full classification, only one more family---namely $\E(3)$---is required.  The generic algebras in the family $\E(3)$ are all isomorphic, and can be obtained using the universal deformation formula of Coll, Gerstenhaber and Giaquinto~\cite{Coll1989} associated with actions of the two-dimensional nonabelian Lie algebra.  Therefore, as we explain in \autoref{sec:schwarz}, we can view the quantization through the lens of equivariant geometry.  We find that the corresponding quantum $\PP[3]$ contains three commutative rational curves---a line, a plane conic, and a twisted cubic---corresponding to various embeddings of $\PP[1]$ in its symmetric power.  In addition to these curves, one also finds a pencil of noncommutative sextic surfaces whose classical limits are the level sets of the $j$-invariant.  Finally, the Schwarzenberger bundles~\cite{Schwarzenberger1961}, which served as early examples of indecomposable vector bundles on projective space, quantize to give graded bimodules over the $\E(3)$ algebra.
\\
 
\emph{Outline:} The paper is organized as follows.  In \autoref{sec:reduction}, we review some basic facts about deformations of quadratic algebras and their semiclassical limits with a view towards noncommutative projective geometry.  We discuss Calabi--Yau algebras, their superpotentials, their role as unique homogeneous coordinate rings for quantum $\PP[n]$s, and their connection with the unimodularity condition for Poisson structures.  Most of this discussion consists of collecting results that are already in the literature, but we hope that the presentation may yet be of some value.  With these ideas in place, all that remains in order to complete the proof of \autoref{thm:qa-cmpts} is to give normal forms for the algebras in \autoref{tab:p3quant} and compare their semiclassical limits with the foliations classified by Cerveau and Lins Neto.  This task is the subject of \autoref{sec:cmpts}.  We close the paper in \autoref{sec:schwarz} with a brief discussion of the equivariant geometry of the $\E(3)$ algebra and related examples.
\\ 

\emph{Acknowledgements:} The author wishes to thank Ragnar-Olaf Buchweitz, Pavel Etingof, Nigel Hitchin, Daniel Rogalski, Travis Schedler, Toby Stafford, Michel Van den Bergh and Chelsea Walton for helpful discussions.  Special thanks are due to Marco Gualtieri, who suggested the problem and provided much support and encouragement as the author's thesis advisor; Colin Ingalls, who provided helpful commentary on the related chapter in the author's PhD thesis~\cite{Pym2013}; and Raf Bocklandt, who shared \texttt{GAP}~\cite{GAP4} code that was used to compute some of the superpotentials presented here.  Further computations were done using \texttt{Sage}~\cite{Sage}, and the three-dimensional renderings in \autoref{fig:CLN} were produced using \texttt{surfex}~\cite{Holzer2008}.

At various stages, this work was supported by an NSERC Canada Graduate Scholarship (Doctoral), the University of Toronto, McGill University and EPSRC Grant EP/K033654/1.  Significant advances occurred during visits to the Mathematical Sciences Research Institute for workshops associated with the program on Noncommutative Algebraic Geometry and Representation Theory.  The author thanks the MSRI and the organizers of the workshops for their support and hospitality.

\section{Reduction to the semi-classical limits}
\label{sec:reduction}

\subsection{Quadratic deformations}

It is common in noncommutative projective geometry (see, e.g., \cite{Artin1990,Artin1994,Stafford2001}) to describe quantum versions of projective space and other projective varieties by their homogeneous coordinate rings.  Just as classical projective space $\PP[n]$ corresponds to the polynomial ring $\CC[x_0,\ldots,x_n]$, graded so that the degree of each generator $x_i$ is equal to one, a quantum projective space should be described by a noncommutative graded algebra $\A = \bigoplus_{k \ge 0} \A_k$ that resembles a polynomial ring as much as possible.  We focus in this paper on the case when $\A$ arises as a deformation of the usual product on the polynomial ring, following the approach of Bondal~\cite{Bondal1993} and Kontsevich~\cite{Kontsevich2003}.

We begin by viewing the polynomial ring as a quotient of the free algebra by a two-sided ideal:
\[
\CC[x_0,\ldots,x_n] = \frac{\CC\abrac{x_0,\ldots,x_n}}{(x_ix_j-x_jx_i)_{0 \le i <j \le n}}
\]
The key point is that the ideal of relations is generated by $r= {n+1\choose 2}$ homogeneous quadratic expressions in $x_0,\ldots,x_n$ that express the commutativity of the variables.  More invariantly, we may consider a vector space $\V$ and present its symmetric algebra as a quotient
\[
\sym{\V} = \tens{\V}/(\ext[2]{\V}).
\]
of the tensor algebra by the ideal generated by the skew-symmetric two-tensors.

We therefore construct quantizations of the homogeneous coordinate ring by deforming the subspace $\ext[2]{\V} \subset \V \otimes \V$ to a new element $\R \in \Gr{r,\V\otimes \V}$ of the relevant Grassmanian.  In other words, we deform $\A$ as a \defn{quadratic algebra}.  Such a deformation will always produce a new graded algebra $\A = \tens{\V}/(\R) = \bigoplus_{k \ge 0} \A_k$ with scalars $\A_0 = \CC$ and generators $\A_1 = \V$, but we wish to select the deformations that have the same ``size'' as the polynomial ring, in the sense that the Hilbert series $H_\A(t) = \bigoplus_{k \ge 0}(\dim \A_k)t^k$ is equal to $(1-t)^{-n-1}$.

\begin{example}
Choose constants $q_{ij}\in \CC^\times$ for $0 \le i < j \le n$.  We can then form the \defn{skew polynomial ring}
\[
\A = \frac{\CC\abrac{x_0,\ldots,x_n}}{(x_ix_j-q_{ij}x_jx_i)_{0\le i<j\le n}}.
\]
The monomials of the form $x_0^{k_0}\cdots x_n^{k_n}$ clearly form a basis for $\A$, and hence $\A$ has the same Hilbert series as the polynomial ring.\qed
\end{example}

If we choose an integer $K \ge 0$ and look at algebras for which $H_\A(t) = (1-t)^{-n-1} \mod t^{K+1}$ (i.e. algebras for which the dimensions are correct up to degree $K$), we are imposing a collection of algebraic equations on the coefficients in the relations, and so we obtain a closed subscheme
\[
\QA(\V,K) \subset \Gr{r,\V\otimes \V}.
\]
The set of all quadratic algebras with the correct Hilbert series is therefore
\[
\QA(\V) = \bigcap_{K \ge 3} \QA(\V,K).
\]
This set is obviously non-empty because it contains the point $Comm_\V = \ext[2]{\V} \in \QA(\V)$ that corresponds to the polynomial ring.  A result of \cite{Drinfeld1992}, known as the Koszul deformation principle (see also \cite[Theorem 2.1]{Polishchuk2005}), implies that for every $K \ge 3$, the subvariety $\QA(\V,K) \subset \QA(\V,3)$ contains an open neighbourhood of $Comm_\V$ in $\QA(\V,3)$.  Hence the germ of $\QA(\V,3)$ at the point $Comm_\V$ can be viewed as the moduli space parametrizing germs of deformations of $\sym{\V}$ as a Koszul quadratic algebra with Hilbert series $(1-t)^{-n-1}$ and fixed vector space $\V$ of generators~\cite{Bondal1993,Kontsevich2001}.  Since $\sym{\V}$ is the ring of functions on the dual vector space $\V^*$, we can think of this germ as parametrizing quantum deformations of the variety $\V^*$ that are compatible with the action of $\CC^\times$ by rescaling.
\begin{definition}
The germ $\QDef(\V)$ of $\QA(\V,3)$ at $Comm_\V$ is the \defn{space of homogeneous quantum deformations of $\V^*$}.
\end{definition}

If $\star_\hbar$ is a family of products on $\CC[x_0,\ldots,x_n]$ that specializes to the usual product when $\hbar = 0$, we can take the semiclassical limit
\[
\{f,g\} = \left.\frac{d}{d\hbar}\right|_{\hbar = 0} (f \star_\hbar g - g \star_\hbar f)
\]
to obtain a Poisson bracket.  Similarly, for quadratic deformations as above, the semiclassical limits of the relations give rise to quadratic Poisson brackets:
\[
\{x_i,x_j\} = \sum_{k,l=0}^n c^{kl}_{ij}x_kx_l
\]
for constants $c^{kl}_{ij} \in \CC$.  Invariantly, the Poisson structure is determined by a linear map $\ext[2]{\V} \to \sym[2]{\V}$ that specifies the bracket on generators.  The Jacobi identity imposes a set of homogeneous quadratic equations on the coefficients $c^{kl}_{ij}$, and so the space of quadratic Poisson structures is given by a subscheme
\[
\PB(\V) \subset \ext[2]{\V^*} \otimes \sym[2]{\V} 
\]
that is invariant under rescaling, i.e.~a cone.

The process of taking semi-classical limits identifies $\PB(\V)$ with the tangent cone of $\QDef(\V)$~\cite{Bondal1993}.  Kontsevich has proved as a consequence of his formality theorem that this tangent cone contains all of the information necessary for a complete description of $\QDef(\V)$:
\begin{theorem}[{\cite[Proposition 6]{Kontsevich2001}}]\label{thm:kont-iso} The identification of $\PB(\V)$ with the tangent cone of $\QDef(\V)$ extends to an isomorphism $Germ(\PB(\V),0) \cong \QDef(\V)$ of analytic germs.
\end{theorem}
In fact, Kontsevich works with formal neighbourhoods rather than analytic germs, and conjectures that the formal isomorphism he constructs is actually analytic.  The status of this conjecture is unclear, but by combining Kontsevich's formal isomorphism with Artin's approximation theorem~\cite{Artin1969}, we can deduce that there exists some analytic isomorphism as in the theorem.  This analytic isomorphism may, in principle, be different from the formal isomorphism constructed by Kontsevich.

\begin{corollary}\label{thm:cmpt-bij}
The process of taking semiclassical limits gives a bijection between the set of irreducible components of $\QDef(\V)$ and those of $\PB(\V)$.
\end{corollary}

\begin{proof}
The statement is true at the level of germs.  It therefore suffices to observe that because $\PB(\V)$ is a cone with vertex the zero Poisson structure, the irreducible components of the germ of $\PB(\V)$ at 0 are in bijection with the irreducible components of $\PB(\V)$ itself.
\end{proof}

\subsection{Descent to projective space and the Calabi--Yau condition}
\label{sec:cy}
The homogeneous coordinate ring of a projective variety depends on the ample line bundle used to embed it in projective space.  For projective space itself, there is a natural choice given by the ample generator of the Picard group, and the result is the polynomial ring.

However, in the noncommutative setting, there is more room for ambiguity.  Namely, because there is no underlying topological space in noncommutative geometry, it is generally understood that a quantum projective variety is best described by its abelian category of quasi-coherent sheaves~\cite{Artin1994}.  This category can be obtained as a quotient of the category of graded left (or right) modules over the homogeneous coordinate ring.  Since different rings may have equivalent module categories, it is perfectly possible, and indeed common, for two graded rings $\A$ to describe the same quantum projective variety $\Proj{\A}$.

This extra equivalence relation is (essentially) implemented by the formation of so-called \defn{Zhang twists}~\cite{Zhang1996}, in which a graded algebra $\A = \bigoplus_{k \ge 0} \A_k$ is twisted by an automorphism $\sigma : \A \to \A$ to obtain a new algebra $\A^\sigma$ whose underlying graded vector space is the same.  The new product $\bullet_\sigma$ is given by the formula
\[
f \bullet_\sigma g = f \sigma^k(g)
\]
for $f \in \A_k$ and $g \in \A$, and a similar formula allows one to twist any graded $\A$-module.  Thus, the process of twisting does not change the category of graded modules and so $\Proj{\A}$ and $\Proj{\A^\sigma}$ are equivalent as quantum varieties.
\begin{example}
The automorphism $\sigma$ of $\CC[x,y]$ that sends $x$ to itself and $y$ to $y-x$ produces the Zhang twist
\[
\CC[x,y]^\sigma \cong \CC\abrac{x,y}/(xy-yx+x^2),
\]
commonly known as the Jordan plane.  Even though this algebra is noncommutative, it should still be viewed as a homogeneous coordinate ring for the classical projective line $\PP[1]$.\qed
\end{example}

In order to remove this ambiguity, one can use the notion of a $\ZZ$-algebra~\cite{Bondal1993a}, or put further constraints on the algebra $\A$.  We do so here using the notion of a Calabi--Yau algebra~\cite{Ginzburg2006} and its twisted version.  This condition is homological in nature and closely related to noncommutative Poincar\'e duality.  Since we work with graded rings, the twisted Calabi--Yau condition is equivalent to AS regularity~\cite[Lemma 1.2]{Reyes2013}, but the (untwisted) Calabi--Yau condition is strictly stronger.  We will not need the precise formulation here; we simply require the
\begin{theorem}[{\cite{Bocklandt2010,Dubois-Violette2005,Dubois-Violette2007}}]
A Koszul algebra $\A$ is Calabi--Yau (respectively, twisted Calabi--Yau) if and only if it is derived from a superpotential (resp., twisted superpotential) whose associated complex is exact.
\end{theorem}
For our purposes, a \defn{superpotential} on a vector space $\V$ of dimension $d$ will be an element $\Phi \in \V^{\otimes d}$  such that $cyc(\Phi) = (-1)^d \Phi$, where $cyc: \V^{\otimes d} \to \V^{\otimes d}$ is the linear automorphism that cyclically permutes the tensor factors.  Similarly, a twisted superpotential $\Phi \in \V^{\otimes d}$ satisfies $cyc(\Phi) = (-1)^d (Q\otimes 1 \otimes \cdots \otimes 1)\Phi$ for some linear automorphism $Q \in \GL{\V}$.  More generally, one can consider superpotentials that are tensors of other degrees, but we shall not do so here.

From a (possibly twisted) superpotential $\Phi$, we obtain a quadratic algebra by ``differentiating'' as follows: denote by $\partial^k\Phi \subset \V^{\otimes(d-k)}$ the subspace given by contracting the first $k$ factors of $\Phi$ with $(\V^*)^{\otimes k}$.  Then $\partial^{d-2}\Phi$ gives a subspace of $\V\otimes \V$, and hence gives the relations for a quadratic algebra $\A$.  The tensor $\Phi$ should be thought of as a volume form on the noncommutative affine variety corresponding to $\A$.  Using the rest of the subspaces $\partial^k\Phi$, which serve as differential forms of lower degree, one obtains the complex of $\A$-bimodules whose exactness is mentioned in the theorem.

\begin{example}[{\cite{Ginzburg2006}}]
For the case $\V \cong \CC^3$ with basis $x,y,z$, consider the tensor
\[
\Phi = a(xyz+yzx+zxy) + b(xzy+zyx+yxz) + c(x^3+y^3+z^3),
\]
where we have omitted the symbol $\otimes$ for the tensor product.  The contraction with the dual basis vector $\cvf{x}$ in the first factor gives the ``derivative''
\[
\cvf{x}\Phi = ayz+bzy + cz^2,
\]
and similarly for $\cvf{y}$ and $\cvf{z}$.  The resulting algebra is the well-known three-dimensional Skylanin algebra 
\[
\A = \frac{\CC\abrac{x,y,z}}{\begin{pmatrix}
axy+byx+cz^2, \\ ayz+nzy+cx^2 , \\ azx+bxz+cy^2
\end{pmatrix}}
\]
of~\cite{Artin1990,Artin1987}, which describes a nontrivial family of quantum deformations of $\PP[2]$.\qed
\end{example}

The key fact about twisted superpotentials that allows us to remove the ambiguity in the homogeneous coordinate rings is that, for a Koszul, twisted Calabi--Yau algebra, the automorphism $Q \in \GL{\V}$ in the definition is unique.   Moreover, given a $d$th root $\sigma$ of $Q$, one can perform a Zhang twist by $\sigma$ to replace $Q$ with the identity~\cite[Sections 3 and 7]{Dubois-Violette2007}; see also \cite{Goodman2013,He2013,Reyes2013} for related constructions.  Since any $Q$ sufficiently close to the identity will have a $d$th root that depends analytically on $Q$, it follows that any sufficiently small analytic deformation of the polynomial ring as a Koszul twisted Calabi--Yau algebra can be Zhang twisted to an analytic family of Koszul Calabi--Yau algebras, and that this family of Calabi--Yau algebras is unique up to isomorphism.  We therefore obtain a canonical choice for the homogeneous coordinate rings of the corresponding quantum projective spaces.  As a result, we may restrict our attention to the deformations that are Calabi--Yau:
\begin{definition}
For a vector space $\V$, the \defn{space of quantum deformations of $\PP(\V^*)$} is the closed subgerm $\QDef_0(\V) \subset \QDef(\V)$ parametrizing deformations of the polynomial ring as a Koszul Calabi--Yau algebra, i.e.~deformations that are derived from a superpotential with no twist.
\end{definition}

Thus our main result (\autoref{thm:qa-cmpts}) describes the irreducible components of $\QDef_0(\V)$ in the case when $\dim \V = 4$.

Dolgushev~\cite{Dolgushev2009} has shown that quantization relates the Calabi--Yau condition for noncommutative algebras to the unimodularity condition for Poisson structures.  Similarly, Shoikhet~\cite{Shoikhet2010} has shown that quantization is compatible with Koszul duality.  Since the superpotential for a Koszul Calabi--Yau algebra can be computed as a product in the Koszul dual algebra, we can combine their results with \autoref{thm:kont-iso} and \autoref{thm:cmpt-bij} to deduce the
\begin{proposition}The following statements hold:
\begin{enumerate}
\item The process of taking semiclassical limits identifies the tangent cone of $\QDef_0(\V)$ at $Comm_\V$ with the subscheme $\PB_0(\V) \subset \PB(\V)$ consisting of those quadratic Poisson structures that are unimodular.
\item This identification extends to an isomorphism $Germ(\PB_0(\V),0) \cong \QDef_0(\V)$ of analytic germs.
\item Taking semiclassical limits gives a bijection between the set of irreducible components of $\QDef_0(\V)$ and those of $\PB_0(\V)$.
\end{enumerate}
\end{proposition}

We recall that a Poisson structure on a manifold $\X$ is \defn{unimodular} if there is a volume form on $\X$ that is invariant under all Hamiltonian flows~\cite{Weinstein1997}.  It is known~\cite{Bondal1993,Liu1992} that a quadratic Poisson structure $\{\cdot,\cdot\} \in \PB(\V)$ can be uniquely decomposed as
\[
\{f,g\} = \{f,g\}_{unim} + Z(f)E(g)-E(f)Z(g)
\]
where $\{\cdot,\cdot\}_{unim}$ is a unimodular quadratic Poisson structure, $E$ is the Euler derivation $E(f) = \sum_{i=0}^n x_i\cvf{x_i}f$ and $Z$ is a homogeneous derivation (a linear vector field) that is an infinitesimal symmetry of $\{\cdot,\cdot\}_{unim}$.

Since a quadratic Poisson structure is invariant under the action of $\CC^\times$ on the manifold $\V^*$, it descends to a Poisson structure on $\PP(\V^*)$.  Conversely, a Poisson structure on projective space can always be lifted to a quadratic Poisson structure~\cite{Bondal1993,Polishchuk1997}.  The decomposition above implies that there is a unique unimodular lift, and that the other lifts differ by a term involving an infinitesimal symmetry $Z$.  Thus $Z$ is the infinitesimal version of the Zhang twist $\sigma$ in the quantum case.  The output of this discussion is the following result, which we shall apply in the next section in order to complete the proof of \autoref{thm:qa-cmpts}:
\begin{theorem}\label{thm:proj-cmpt-bij}
Taking semiclassical limits gives a bijection between the set of irreducible components of $\QDef_0(\V)$ and the set of irreducible components of the variety parametrizing Poisson structures on $\PP(\V^*)$.
\end{theorem}

\section{Deformations of $\PP[3]$}

\label{sec:cmpts}

We now focus our attention on the quantum deformations of $\PP[3]$, i.e.~the case in which the vector space $\V$ of generators of the algebra has dimension four.  Throughout this section $x_0,\ldots,x_3$ will be a basis for $\V$.  In this case, one can show using the Batalin--Vilkovisky formalism~\cite{Polishchuk1997,Xu1999} that any unimodular quadratic Poisson structure may be written uniquely as
\begin{equation}
\{f,g\} = \frac{df\wedge dg \wedge d\alpha}{dx_0\wedge dx_1 \wedge dx_2 \wedge dx_3}\label{eqn:pot-brac}
\end{equation}
for a one-form $\alpha = \alpha_0dx_0 + \alpha_1dx_1+\alpha_2 dx_2 +  \alpha_3 dx_3$ that satisfies the integrability condition $\alpha \wedge d\alpha = 0$, and whose coefficients are homogeneous cubic polynomials satisfying $\sum_{i=0}^3 x_i\alpha_i = 0$.

At the level of the projective space $\PP[3]$, this one-form can be viewed as an integrable section $\alpha' \in \cohlgy[0]{\PP[3],\forms[1]{\PP[3]}(4)}$; the Poisson bivector field on $\PP[3]$ corresponding to $\{\cdot,\cdot\}$ is obtained through the identification $\forms[1]{\PP[3]}(4) \cong \ext[2]{\tshf{\PP[3]}}$ induced by the volume form $dx_0\wedge dx_1\wedge dx_2\wedge dx_3$.  In particular, the symplectic leaves of the Poisson structure on $\PP[3]$ are the leaves of the foliation defined by the kernel of $\alpha'$.  Remarkably, the irreducible components in the space parametrizing such foliations have been completely described in celebrated work of Cerveau and Lins Neto:
\begin{theorem}[\cite{Cerveau1996}]
For $n \ge 3$, the variety parametrizing integrable global sections of $\forms[1]{\PP[n]}(4)$ has exactly six irreducible components, called $\bL(1,1,1,1)$, $\bL(1,1,2)$, $\R(2,2)$, $\R(1,3)$, $\bS(2,3)$ and $\E(3)$.  Moreover, there are explicit formulae for the cubic one-forms that define each component.
\end{theorem}
Strictly speaking, Cerveau and Lins Neto dealt with the case of foliations whose singular set has no components of dimension larger than one, but recent work by Loray, Pereira and Touzet~\cite{Loray2013} shows that all of the other foliations corresponding to Poisson structures lie in the closures of the families considered by Cerveau and Lins Neto.  Moreover, they described similar results for the other Fano threefolds of Picard rank one.  The components $\R(2,2)$ and $\R(1,3)$ were also examined by Polishchuk~\cite{Polishchuk1997}, where they were identified as giving the only examples of Poisson structures on $\PP[3]$ whose zero locus contains a smooth curve as a connected component.

In what follows, we use the explicit description of the components in the space of foliations to give similar descriptions of the Poisson brackets.  For each component, we select a Zariski open set of suitably generic foliations and put the corresponding one-forms into a normal form.  We then compute the brackets via \eqref{eqn:pot-brac}, obtaining normal forms for the unimodular quadratic Poisson structures.  We therefore have an explicit description of the irreducible components in $\PB_0(\V)$ as per \autoref{thm:pb-cmpts}.

With these normal forms for the Poisson brackets in hand, we are able to find normal forms for the corresponding quantizations, and to deduce that they are Koszul and Calabi--Yau with Hilbert series $(1-t)^{-4}$.  The Koszul condition is easily checked since most of the algebras are already known.  In order to prove that they are Calabi--Yau, we simply exhibit their superpotentials and appeal to the following
\begin{lemma}
Let $\A = \tens{\V}/(\R)$ be a Koszul algebra that  is derived from a superpotential $\Phi$ and has the same Hilbert series as the polynomial ring in four variables.  If $\partial^{3}\Phi = \V$, then $\A$ is Calabi--Yau.
\end{lemma}
\begin{proof}
See the proof of \cite[Proposition 7.1]{Bocklandt2010}, where the lemma is used implicitly to prove that the Sklyanin algebras are Calabi--Yau.
\end{proof}

In this way, we obtain six irreducible families of flat deformations of $\sym{\V}$ as a Koszul Calabi--Yau algebra.  The semiclassical limits of these six families are the given six families of unimodular Poisson structures, so it follows from \autoref{thm:proj-cmpt-bij} that these deformations describe all of the irreducible components of $\QDef_0(\V)$, proving \autoref{thm:qa-cmpts}.

In writing the superpotentials, we use the symbol $\circlearrowright$ to denote a sum over cyclic permutations of tensor factors, weighted by alternating signs.  Thus, for example,
\begin{align*}
x_0x_1x_2^2\ + \circlearrowright &= x_0x_1x_2^2 - x_1x_2^2x_0 + x_2^2x_0x_1 - x_2x_0x_1x_2.
\end{align*}
We also use the symbols $[\cdot]^{+}$ and $[\cdot]^{-}$ to denote symmetrization and anti-symmetrization of monomials.  Thus, for example
\[
[x_0x_1x_2]^{+} = x_0x_1x_2+x_0x_2x_1+x_1x_0x_2+x_1x_2x_0+x_2x_0x_1+x_2x_1x_0
\]
and
\[
[x_0x_1x_2]^{-} = x_0x_1x_2-x_0x_2x_1-x_1x_0x_2+x_1x_2x_0+x_2x_0x_1-x_2x_1x_0.
\]
For brevity, we include only the details about the Poisson structures and their quantizations that are necessary for the proof of \autoref{thm:qa-cmpts}.  We have also omitted several routine calculations, many of which were performed with the aid of a computer.  

\subsection{The $\bL(1,1,1,1)$ component}
\label{sec:l1111}

\subsubsection*{Poisson structure}

Consider the family of cubic one-forms on $\V^*$ having of the form
\[
\alpha = x_0x_1x_2x_3 \sum_{i=0}^3 a_i \frac{dx_i}{x_i},
\]
where $a_0,\ldots,a_3 \in \CC$ are constants such that $\sum_{i=0}^3 a_i = 0$.  The corresponding family of Poisson brackets has the form
\begin{equation}
\begin{aligned}
\{x_i,x_{i+1}\} &= (-1)^i(a_{i+3}-a_{i+2})x_ix_{i+1} \\
\{x_i,x_{i+2}\} &= (-1)^i(a_{i+1}-a_{i+3})x_ix_{i+2}
\end{aligned}\label{eqn:l1111-brac}
\end{equation}
where the indices are taken modulo four.   The closure of the $\GL{\V}$-orbit of this family is the component $\bL(1,1,1,1)$ in the space $\PB_0(\V)$ of unimodular quadratic Poisson structures.

\subsubsection*{Quantization}
From the form of the Poisson brackets, it is clear that the quantizations should be skew polynomial rings---that is, algebras with four degree-one generators $x_0,x_1,x_2,x_3$ and quadratic relations $x_ix_j = p_{ij}x_jx_i $, where $p_{ij}=p_{ji}^{-1}\in \CC^\times$.  Any such algebra is Koszul, but only certain choices of the constants $p_{ij}$ will produce a Calabi--Yau algebra~\cite[Example 5.5]{Reyes2013}, namely those for which the relations may be written
\begin{equation}
\begin{aligned}
x_ix_{i+1} &=  \rbrac{\tfrac{q_{i+3}}{q_{i+2}}}^{(-1)^i}x_{i+1}x_i \\
x_ix_{i+2} &= \rbrac{\tfrac{q_{i+1}}{q_{i+3}}}^{(-1)^i}x_{i+2}x_i
\end{aligned}\label{eqn:l1111-rels}
\end{equation}
where $q_0,\ldots,q_3 \in \CC^*$ satisfy $\prod_{i=0}^3q_i = 1$, and again the indices are taken modulo four.  This algebra is given by the following superpotential:
\begin{align*}
\Phi_{\bL(1,1,1,1)} &= \tfrac{q_0q_2}{q_1q_3}x_0x_1x_2x_3 - \tfrac{q_2}{q_3}x_0x_1x_3x_2 - \tfrac{q_2}{q_1}x_0x_2x_1x_3 \\ & \ \  + \tfrac{q_0}{q_1}x_0x_2x_3x_1 + \tfrac{q_0}{q_3}x_0x_3x_1x_2-x_0x_3x_2x_1 + \circlearrowright
\end{align*}
where $\circlearrowright$ denotes a supercyclic sum---the sum over cyclic permutations of the given expression, with appropriate signs.  

If we set $q_i = e^{\hbar a_i}$ for $0 \le i \le 3$, then the Poisson brackets \eqref{eqn:l1111-brac} are obtained as the semi-classical limit $\hbar \to 0$ of the relations \eqref{eqn:l1111-rels}.  This product is essentially the Moyal--Vey quantization~\cite{Moyal1949,Vey1975}.

\subsection{The $\bL(1,1,2)$ component}

\subsubsection*{Poisson structure}

Choose a homogeneous quadratic polynomial $g \in \sym[2]{\V}$ and consider the family of cubic one-forms
\[
\alpha = x_0x_1g\rbrac{a_0\frac{dx_0}{x_0} + a_1\frac{dx_1}{x_1} + b\frac{dg}{g}},
\]
where $a_0,a_1,b \in \CC$ are constants such that $a_0+a_1+2b=0$.

Generically, we can put this data into a normal form as follows.  We make the assumption that $g$ is non-degenerate, and that the dual quadratic form $g^{-1} \in \sym[2]{\V^*}$ restricts to a nondegenerate form on the span of $x_0$ and $x_1$; this assumption defines a Zariski open set in the space of quadratic forms.  By adjusting the coefficients $a_0,a_1$ and $b$, we may rescale $x_0$ and $x_1$ so that the inner products are $g^{-1}(x_0,x_0) = g^{-1}(x_1,x_1) = 1$ and $g^{-1}(x_0,x_1) = \lambda \in \CC$, where $\lambda$ is arbitrary.  We may then choose our remaining basis vectors $x_2,x_3$ so that
\[
g = x_0^2 + \tfrac{\lambda}{2}x_0x_1 + x_1^2 + x_2x_3. 
\]
Computing the bracket from the one-form using \eqref{eqn:pot-brac} and setting $c_0 = a_0-b$ and $c_1 = a_1-b$ gives the normal form for the Poisson brackets:
\begin{equation}
\begin{aligned}
\{x_0,x_1\} & =0  & 
\{x_2,x_3\} &=  (c_0-c_1) \left(x_{0}^{2} + \lambda  x_{0} x_{1} + x_{1}^{2} + x_{2} x_{3}\right) \\   
\{ x_0 , x_2 \} &=  c_0x_0x_2  &  \{ x_1 , x_2 \} &=  -c_1x_1x_2 \\
\{ x_0 , x_3 \} &=  -c_0x_0x_3 &  \{x_1,x_3\} &= c_1x_1x_3,  \\
\end{aligned}\label{eqn:l112-brac}
\end{equation}
with $c_0,c_1,\lambda \in \CC$.  The component $\bL(1,1,2) \subset \PB_0(\V)$ in the space of unimodular quadratic Poisson brackets is the closure of the $\GL{\V}$-orbit of this normal form.
%
\subsubsection*{Quantization}

Most of the Poisson brackets in \eqref{eqn:l112-brac} resemble the brackets from the $\bL(1,1,1,1)$ case, for which the quantization is a skew-polynomial ring.  We therefore look for Calabi--Yau algebras with relations of the form
\begin{equation}
\begin{aligned}
x_1x_0 &= x_0x_1 			& x_3x_2 &= p_0^{-1}p_1\,x_2x_3 + F\\
x_2x_0 &= p_0^{-1}\,x_0x_2 	& x_3x_0 &= p_0\,x_0x_3 \\
x_2x_1 &= p_1\,x_1x_2 		& x_3x_1 &= p_1^{-1}\,x_1x_3 \\
\end{aligned}\label{eqn:l112-rels}
\end{equation}
with $F$ a quadratic polynomial in $x_0$ and $x_1$.  These relations are readily seen to be of the form described by Cassidy, Goetz and Shelton in~\cite[Theorem 1.1]{Cassidy2006}.  In particular, they define a Koszul, Artin--Schelter regular algebra with Hilbert series $(1-t)^{-4}$.  A computer calculation shows that this algebra will be Calabi--Yau exactly when 
\[
F = (p_1-p_0)(x_0^2+\lambda x_0x_1+x_1^2) + (1-p_0^2)x_0^2+(p_1^2-1)x_1^2
\]
for some $\lambda \in \CC$, in which case the algebra is determined by the superpotential
\begin{align*}
\Phi_{\bL(1,1,2)} &= ( 1-p_0+p_1-p_0^2 ) x_0^3x_1   + \tfrac{\lambda}{2}( p_1-p_0) x_0x_1x_0x_1  \\
&\ \ \ + ( p_1^2-p_0+p_1-1 ) x_0x_1^3  + p_1p_0^{-1}\, x_0x_1x_2x_3 - x_0x_1x_3x_2  \\
&\ \  \ - p_0^{-1} x_0x_2x_1x_2 + p_1p_0^{-1}\, x_0x_2x_3x_1 +  p_1\, x_0x_3x_1x_2 - x_0x_3x_2x_1 \\
&\ \ \ + \circlearrowright.
\end{align*}
Setting $p_i = e^{\hbar c_i}$ for $i=0,1$ and taking the semi-classical limit $\hbar \to 0$ recovers the Poisson brackets \eqref{eqn:l112-brac}.

\subsection{The $\bR(2,2)$ component}

\subsubsection*{Poisson structure}
Given a pair of homogeneous quadratic forms $g_1,g_2 \in \sym[2]{\V}$ we may define a cubic one-form on $\V^*$ by
\[
\alpha = g_1\, dg_2 - g_2\, dg_1.
\]
Such one-forms define a Zariski open set in the irreducible component $\R(2,2)$ of the space of degree-two foliations on $\PP(\V^*)$.  We now seek a normal form for the Poisson brackets.

Clearly $\alpha$ depends only on the element $g_1\wedge g_2 \in \ext[2]{(\sym[2]{\V^*})}$, or up to scale, on the pencil of quadrics in $\PP(\V^*)$ spanned by $g_1$ and $g_2$.  Recall that every sufficiently generic pencil of quadrics is equivalent to one in which $g_1$ and $g_2$ have the form
\begin{align*}
g_1 &= x_1^2+x_2^2+x_3^2 \\
g_2 &= x_0^2+a_1x_1^2+a_2x_2^2+a_3x_3^2
\end{align*}
with $a_1,a_2,a_3 \in \CC$ satisfying $a_1 + a_2 + a_3 = 0$.  Using \eqref{eqn:pot-brac}, we arrive at the Poisson brackets
\begin{equation}
\begin{aligned}
\{x_0,x_1\} &= (a_3-a_2) x_2x_3 	& \ \ \ \ \{x_2,x_1\} &= x_0x_3 \\
\{x_0,x_2\} &= (a_1-a_3) x_3x_1		& \ \ \ \ \{x_3,x_2\} &= x_0x_1 \\
\{x_0,x_3\} &= (a_2-a_1) x_1x_2		& \ \ \ \ \{x_1,x_3\} &= x_0x_2,
\end{aligned}
\end{equation}
giving the famous Skylanin Poisson structure~\cite{Sklyanin1982}.
 
We remark that every scalar multiple of this Poisson structure is equivalent to another one of the same form via the substitution $x_0 \to cx_0$ with $c\in \CC^*$.  Hence the closure of the $\GL{\V}$-orbit of this normal form defines the component $\R(2,2) \subset \PB_0(\V)$ in the space of unimodular quadratic Poisson structures.

\subsubsection*{Quantization}

The quantizations of the generic Poisson structures in this component are the Sklyanin algebras~\cite{Sklyanin1982}, which appear in the same paper as the Poisson brackets.  These algebras have been well studied; for example, they have Hilbert series $(1-t)^{-4}$ and are Koszul and Calabi--Yau~\cite{Bocklandt2010,Dubois-Violette2007,Smith1992,Tate1996},  Moreover, various modules have been constructed that correspond to the projective geometry of the elliptic curve that is the base locus of the pencil of quadrics~\cite{Levasseur1993,Staniszkis1996}.  For completeness, we recall here the formulae for the relations and the superpotential from \cite[Section 7]{Bocklandt2010}.  The elements
\begin{equation}
\begin{aligned}[]
r_1 &= x_0x_1-x_1x_0 - q_1(x_2x_3+x_3x_2) 	&\   s_1 &= x_0x_1+x_1x_0 - (x_2x_3-x_3x_2) \\
r_2 &= x_0x_2-x_2x_0 - q_2(x_1x_3+x_3x_1) 	&\   s_2 &= x_0x_2+x_2x_0 - (x_3x_1-x_1x_3) \\
r_3 &= x_0x_3-x_3x_0 - q_3(x_1x_2+x_2x_1) 	&\   s_3 &= x_0x_3+x_3x_0 - (x_1x_2-x_2x_1)
\end{aligned}\label{eqn:r22-rels}
\end{equation}
in $\V \otimes \V$ give a basis for the space of relations, provided that the tuple of constants $(q_1,q_2,q_3) \in \CC^3$ satisfies the equation
\[
q_1+q_2+q_3 + q_1q_2q_3 = 0
\]
and is not of the form $(q_1,-1,1)$, $(1,q_2,-1)$ or $(-1,1,q_3)$.

The corresponding superpotential is
\[
\Phi_{\R(2,2)} = \kappa_1(r_1s_1+s_1r_1) + \kappa_2(r_2s_2+s_2r_2) + \kappa_3(r_3s_3+s_3r_3),
\]
where the constants $\kappa_1,\kappa_2,\kappa_3$ are determined by the equations
\[
\kappa_i(1+q_i) = \kappa_{i-1}(1-q_{i-1}),
\]
with the indices $i,i-1 \in \{1,2,3\}$ taken modulo three.

\subsection{The $\bR(1,3)$ component}

\subsubsection*{Poisson structure}
Given $g \in \sym[3]{\V}$ and $x_0 \in \V$, we obtain a cubic one-form on $\V^*$ by the formula
\[
\alpha = 3g\,dx_0-x_0\,dg.
\]
Such one-forms define the irreducible component $\bR(1,3)$ of $\PB_0(\V)$.

We obtain a normal form for the Poisson brackets as follows.  On the hyperplane $x_3 = 0$ in $\V^*$, we may choose  coordinates $x_1,x_2,x_3 \in \V$ and parameters  $\nu,\lambda \in \CC$ so that $g$ has the Hesse form
\[
g = \tfrac{\nu}{3}(x_1^{3}+x_2^3+x_3^3) - \lambda x_1x_2x_3 \mod x_0.
\]
Thus $g$ is given by
\[
g = \tfrac{\nu}{3}(x_1^{3}+x_2^3+x_3^3) - \lambda x_1x_2x_3 + Q(x_1,x_2,x_3)x_0 + L(x_1,x_2,x_3)x_0^2 + Cx_0^3
\]
where $L$ and $Q$ homogeneous of degrees $1$ and $2$, respectively, and $C \in \CC$.  Provided that $Q$ and $L$ are suitably generic, a coordinate change of the form $x_i \mapsto x_i + t_i x_0$ for $1\le i \le 3$ allows us to assume that $L = 0$, i.e., that $g$ has the simpler form
\[
g = \tfrac{\nu}{3}(x_1^{3}+x_2^3+x_3^3) - \lambda x_1x_2x_3 + Q'(x_1,x_2,x_3)x_0+ Cx_0^3.
\]
(A similar argument was used in \cite{LeBruyn1996} to find normal forms in the quantum case.)  Since the Poisson structure only depends on
\[
d\alpha = 4 dg \wedge dx_0
\]
we may subtract $Cx_0^3$ from $g$ without changing the Poisson structure.  Hence, we may assume that
\[
g = \tfrac{\nu}{3}(x_1^{3}+x_2^3+x_3^3) - \lambda x_1x_2x_3 + Q(x_1,x_2,x_3)x_0
\]
without loss of generality.  Let us write
\[
Q(x_1,x_2,x_3) = \tfrac{1}{2}\sum_{i=1}^3\sum_{j=1}^3 b_{ij}x_ix_j
\]
for constants $b_{ij} \in \CC$ with $b_{ij} = b_{ji}$.  Using \eqref{eqn:pot-brac}, we obtain the following normal form for the corresponding Poisson brackets:
\begin{equation}
\begin{aligned}
\{x_0,x_1\} &= 0 & \ \ \ \ \ \  \{x_{2},x_{1}\} &= \lambda x_3x_1 - \nu x_3^2 - \sum_{j=1}^3b_{1j}x_jx_0 \\
\{x_0,x_2\} &= 0 & \ \ \ \ \ \  \{x_{3},x_{2}\} &= \lambda x_3x_2 - \nu x_1^2 - \sum_{j=1}^3b_{2j}x_jx_0 \\
\{x_0,x_3\} &= 0 & \ \ \ \ \ \  \{x_{1},x_{3}\} &= \lambda x_1x_3 - \nu x_2^2 - \sum_{j=1}^3b_{3j}x_jx_0 \\
\end{aligned}\label{eqn:r13-brac}
\end{equation}
with $\nu,\lambda \in \CC$ and $b = (b_{ij}) \in \CC^{3\times 3}$ a symmetric matrix.

\subsubsection*{Quantization}

Since $x_0$ is a Casimir (central) element for the Poisson bracket on $\CC[x_0,x_1,x_2,x_3]$, the bracket is a central extension of a Poisson bracket on $\CC[x_1,x_2,x_3]$.  Correspondingly, the quantization should be a central extension of an Artin--Schelter regular algebra with Hilbert series $(1-t)^{-3}$.  Such algebras were studied in \cite
{LeBruyn1996}, where they are shown to be Koszul and one finds (essentially) the following form for the relations:
\begin{equation}
\begin{aligned}[]
[x_0,x_1] &= 0 & \ \ \ \ \ \  [x_{2},x_{1}] &= \lambda (x_2x_1+x_1x_2) - \nu x_3^2 - \sum_{j=1}^3b_{1j}x_jx_0 \\
[x_0,x_2] &= 0 & \ \ \ \ \ \  [x_{3},x_{2}] &= \lambda (x_3x_2+x_2x_3) - \nu x_1^2 - \sum_{j=1}^3b_{2j}x_jx_0 \\
[x_0,x_3] &= 0 & \ \ \ \ \ \  [x_{1},x_{3}] &= \lambda (x_1x_3+x_3x_1) - \nu x_2^2 - \sum_{j=1}^3b_{3j}x_jx_0, \\
\end{aligned}\label{eqn:r13-rels}
\end{equation}
for constants $\lambda,\nu \in \CC$ and a symmetric matrix $b=(b_{ij}) \in \CC^{3\times 3}$.  Clearly, the semi-classical limit of this family of noncommutative algebras, given by sending $\lambda,\nu,b \to 0$, is the Poisson structure~\eqref{eqn:r13-brac}.

We observe that these algebras are Calabi--Yau, being given by the following superpotential:
\begin{align*}
\Phi_{\R(1,3)} &= \tfrac{1}{4}[x_0 x_1 x_2 x_3]^{-} + \nu\,x_0(x_1^3+x_2^3+x_3^3) - \lambda \, x_0[x_1x_2 x_3]^{+}  \\
&\ \ \ + \tfrac{1}{2}x_0\sum_{i=1}^n\sum_{j=1}^n b_{ij}[x_0x_ix_j]^{+} + \circlearrowright.
\end{align*}

\subsection{The $\bS(2,3)$ component}

\subsubsection*{Poisson structure}

The one-forms that define the component $\bS(2,3)$ are those that are pulled back by a linear projection $\V^* \to \CC^3$.  Thus, in appropriate coordinates, they have the form
\[
\alpha = f_1\,dx_1 +  f_2\,dx_1 +  f_3\,dx_1,
\]
where the homogeneous cubic polynomials $f_1,f_2,f_3$ depend only on $x_1$, $x_2$ and $x_3$.

Thus the derivative has the form
\[
d\alpha = g_1\,dx_2\wedge dx_3 + g_2\, dx_3\wedge dx_1 + g_3\, dx_1\wedge dx_2
\]
for quadratic functions $g_1,g_2,g_3$.  Computing the Poisson brackets using \eqref{eqn:pot-brac}, we see that $x_1,x_2$ and $x_3$ must pairwise Poisson commute, and the remaining Poisson brackets $q_1 = \{x_0,x_1\}$, $q_2 = \{x_0,x_2\}$ and $q_3 = \{x_0,x_3\}$ depend only on $x_1$, $x_2$ and $x_3$.  In other words, this Poisson algebra is an Ore extension of the trivial Poisson structure on $\CC[x_1,x_2,x_3]$ by the derivation
\[
X = g_1\cvf{x_1} + g_2\cvf{x_2} + g_3\cvf{x_3}.
\]
We note that the unimodularity of the Poisson bracket is equivalent to the requirement that this vector field $X$ be divergence-free.

If we projectivize $\CC^3$, the derivation $X$ corresponds to an $\sO{\PP[2]}(1)$-valued vector field $Z \in \cohlgy[0]{\PP[2],\der[1]{\PP[2]}(1)}$ via the Euler sequence on $\PP[2]$.  Since $c_2(\der[1]{\PP(2)}(1))=7$, such a section will vanish at exactly 7 distinct points in $\PP[2]$, provided that $X$ is suitably generic; in fact, the section is determined up to rescaling by these seven points~\cite{Campillo1999}.  Applying a linear automorphism, we may assume that three of these points are $[1,0,0]$, $[0,1,0]$ and $[0,0,1]$.  Correspondingly, $X \wedge E$ must vanish on the lines through $(1,0,0)$, $(0,1,0)$ and $(0,0,1)$ in $\CC^3$.  Here, $E = x_1\cvf{x_1} + x_2\cvf{x_2} + x_3\cvf{x_3}$ is the Euler vector field.  With this constraint, one can readily compute that the components of $X$ must have the form
\[
q_i = a_i x_i^2 + x_i(b_i x_{i+1} + c_ix_{i-1}) + d_ix_ix_{i-1}
\]
for constants $a_i,b_i,c_i,d_i \in \CC$ with the index $i \in \{1,2,3\}$ taken modulo three.  The vector field $X$ is divergence-free if and only if
\[
2a_i + b_{i-1} + c_{i-2} = 0
\]
for all $i$.  Assuming $a_i \ne 0$ for all $i$, applying the transformation $x_i \mapsto a_i^{-1}x_i$ and relabelling the other parameters, we arrive at the following normal form for the Poisson brackets:
\begin{equation}
\begin{aligned}
\{x_0,x_1\} &= x_1^2 + x_1(b_1 x_{2} + c_1x_{3}) + d_1x_2x_{3} &\ \ \ \ \ \  \{x_2,x_3\} &= 0 \\
\{x_0,x_2\} &= x_2^2 + x_2(b_2 x_{3} + c_2x_{1}) + d_2x_3x_{1} &\ \ \ \ \ \ \{x_3,x_1\} &= 0 \\
\{x_0,x_3\} &= x_3^2 + x_3(b_3 x_{1} + c_3x_{2}) + d_3x_1x_{2} &\ \ \ \ \ \ \{x_1,x_2\} &= 0 \\
\end{aligned}\label{eqn:s23-brac}
\end{equation}
where $b_i,c_i,d_i \in \CC$ for $1 \le i \le 3$ satisfy
\[
b_i + c_{i-1} = -2
\]
for all $i$.  The closure of the $\GL{\V}$-orbit of this family gives the component $\bS(2,3) \subset \PB_0(\V)$ in the space of unimodular quadratic Poisson structures.

\subsubsection*{Quantization}

Corresponding to the fact that the Poisson structure~\eqref{eqn:s23-brac} is a Poisson Ore extension of the trivial Poisson structure on $\CC[x_1,x_2,x_3]$, its quantization is a graded Ore extension of the polynomial ring and is therefore Koszul~\cite{Cassidy2008}.

The relations for the quantization take the same form as the Poisson brackets:
\begin{equation}
\begin{aligned}[]
[x_0,x_1] &= x_1^2 + x_1(b_1 x_{2} + c_1x_{3}) + d_1x_2x_{3} &\ \ \ \ \ \ [x_2,x_3] &= 0 \\
[x_0,x_2] &= x_2^2 + x_2(b_2 x_{3} + c_2x_{1}) + d_2x_3x_{1} &\ \ \ \ \ \ [x_3,x_1] &= 0 \\
[x_0,x_3] &= x_3^2 + x_3(b_3 x_{1} + c_3x_{2}) + d_3x_1x_{2} &\ \ \ \ \ \ [x_1,x_2] &= 0 \\
\end{aligned}\label{eqn:s23-rels}
\end{equation}
where $b_i,c_i,d_i \in \CC$ for $1 \le i \le 3$ satisfy
\[
b_i + c_{i-1} = -2
\]
for all $i$.

Define the tensors $Q_1,Q_2,Q_3 \in \V^{\otimes 2}$ by
\begin{align*}
Q_i &= \tfrac{1}{8}(3b_i+c_{i-1})(x_ix_{i+1}+x_{i+1}x_i) + \tfrac{1}{8}(3c_i+b_{i+1})(x_ix_{i+2}+x_{i+2}x_i) \\
&\ \ \ + \tfrac{1}{4}d_i(x_{i+1}x_{i+2}+x_{i+2}x_{i+1})
\end{align*}
with the indices $i,i+1,i+2 \in \{1,2,3\}$ taken modulo $3$.
Then the following tensor is a superpotential for these algebras:
\begin{align*}
\Phi_{\bS(2,3)} &= \frac{1}{4}[x_0x_1x_2x_3]^{-} + Q_1(x_3x_2-x_2x_3) \\
&\ \ \ + Q_2(x_1x_3-x_3x_1)+Q_3(x_2x_1-x_1x_2) + \circlearrowright.
\end{align*}
We conclude that these algebras are Calabi--Yau.

\subsection{The $\E(3)$ component}
\label{sec:e3}

\subsubsection*{Poisson structure}
The final component in the classification is the one that Cerveau and Lins--Neto call the exceptional component $\E(3)$.  The Poisson structure is defined as follows: in a basis $x_0,x_1,x_2,x_3 \in \V$, consider the vector fields
\begin{equation}
\begin{aligned}
X &= -\tfrac{5}{4}x_0\cvf{x_0} - \tfrac{1}{4}x_1\cvf{x_1}+\tfrac{3}{4}x_2\cvf{x_2} + \tfrac{7}{4}x_3\cvf{3} \\
Y &= 4x_0\cvf{x_1}+4x_1\cvf{x_2}+4x_2\cvf{x_3}.
\end{aligned}\label{eqn:e3-vfs}
\end{equation}
on $\V^*$.  These vector fields satisfy the identity
\[
[Y,X] = Y
\]
and, as a result, the formula
\[
\{f,g\} = X(f)Y(g) - Y(f)X(g)
\]
for all $f,g \in \sym{\V}$ defines a Poisson bracket.  The elementary brackets are
\begin{equation}
\begin{aligned}
\{x_0,x_1\} &= 5x_{0}^{2}	& \ \ \ \ \ \{x_1,x_2\} &= x_1^2 + 3x_0x_2 \\
\{x_0,x_2\} &= 5x_0x_1 		& \ \ \ \ \ \{x_1,x_3\} &= x_1x_2+7x_0x_3 \\
\{x_0,x_3\} &= 5x_0x_2 		& \ \ \ \ \ \{x_2,x_3\} &= 7x_1x_3-3x_2^2 \\
\end{aligned}\label{eqn:e3-brac}
\end{equation}
This Poisson structure is rigid: any unimodular quadratic Poisson bracket that is sufficiently close to this one differs from it by a linear change of variables and therefore results in an isomorphic graded Poisson algebra.  The component $\E(3) \subset \PB_0(\V)$ in the space of unimodular quadratic Poisson structures is the closure of its $\GL{\V}$-orbit and contains nontrivial degenerations of this normal form.

\subsubsection*{Quantization}

The quantization of the Poisson bracket \eqref{eqn:e3-brac} can be obtained using the universal deformation formula of Coll, Gerstenhaber and Giaquinto~\cite{Coll1989}, which we presently recall.

Let $\A$ be a commutative $\CC$-algebra and let $X,Y : \A \to \A$ be derivations such that $[Y,X]=Y$.  Define a map
\[
\star : \A[[\hbar]] \otimes_\CC \A[[\hbar]] \to \A[[\hbar]]
\]
by the formula
\begin{align}
f \star g = \sum_{k=0}^\infty \hbar^kY^k(f)\cdot {X \choose k}(g) \label{eqn:cgg}
\end{align}
where
\[
{X \choose k} = \frac{1}{k!}X(X-1)\cdots(X-k+1).
\]
The result of~\cite{Coll1989} is that $\star$ defines an associative product on $\A[[\hbar]]$ whose semiclassical limit $\hbar \to 0$ is the Poisson bracket
\[
\{f,g\} = X(f)Y(g)- Y(f)X(g).
\]
As it stands, the power series is merely formal, while we seek actual, convergent deformations; we wish to set $\hbar = 1$.  Luckily, there is a useful criterion that can be used to guarantee convergence.  Recall that a derivation of $Z : \A \to \A$ is \defn{locally nilpotent} if for every $a \in \A$ there exists a $k \in \NN$ such that $Z^k(a) = 0$.  The following observation is immediate from the CGG formula \eqref{eqn:cgg}:
\begin{lemma}\label{lem:loc-nilp}
If the derivation $Y \in \g$ acts locally nilpotent on $\A$, then the CGG formula for $f \star g$ truncates to a polynomial in $\hbar$ for any $f,g \in \A$, i.e., it defines a map
\[
\A \otimes_\CC \A \to \A[\hbar].
\]
Evaluation at a particular value $\hbar \in \CC$ then gives an associative product
\[
\star_\hbar : \A \otimes_\CC \A \to \A.
\]
\end{lemma}
It is straightforward to show using the commutation relation $[Y,X]=Y$ that for any $\hbar \ne 0$, the $\CC$-linear automorphism $e^{\hbar X} : \A \to \A$ gives an isomorphism between the products $\star_{\hbar}$ and $\star_1$, and hence the quantizations for different values of $\hbar \ne 0$ are all isomorphic.

For the particular case of the $\E(3)$ Poisson structure, it is evident from \eqref{eqn:e3-vfs} that $Y$ is locally nilpotent and so the lemma applies.  Computing the commutators and setting $\hbar = 1$, we obtain the algebra with generators $x_0,\ldots,x_3$ and relations
\begin{equation}
\begin{aligned}[]
[x_0,x_1]&= 5x_0^2 \\
[x_0,x_2]&=-\tfrac{45}{2}x_0^2 + 5x_0x_1 \\
[x_0,x_3]&= \tfrac{195}{2}x_0^2-\tfrac{45}{2}x_0x_1 + 5x_0x_2 \\
[x_1,x_2]&= -\tfrac{3}{2}x_0x_1 + 3x_0x_2+x_1^2 \\
[x_1,x_3]&=5x_0x_1-3x_0x_2 + 7x_0x_3-\tfrac{5}{2}x_1^2+x_1x_2 \\
[x_2,x_3]&= -\tfrac{77}{2}x_0x_2-\tfrac{77}{2}x_0x_3+\tfrac{21}{2}x_1x_2+7x_1x_3-3x_2^2. \\
\end{aligned}
\label{eqn:e3-rels}
\end{equation}
The sequence $(x_0,x_1,x_2,x_3)$ is a normal regular sequence, and hence this algebra is Koszul.  It is derived from the following superpotential:
\begin{align*}
\Phi_{\E(3)} &= \tfrac{75}{2}  x_0^3 x_1  -100   x_0^3 x_2  -25 
  x_0^3 x_3  - \tfrac{75}{4} x_0 ^2 x_1^2 +  30  x_0^2 x_1 x_2 + \tfrac{15}{2} x_0^2 x_1 x_3  \\
&\ \ \ +  50   x_0^2 x_2 x_1  -15 x_0^2 x_2^2  -5  x_0^2 x_2 x_3  + \tfrac{55}{2}  x_0^2 x_3 x_1 +
  5  x_0^2 x_3 x_2 - 6  x_0 x_1 x_0 x_1  \\
&\ \ \ +  13  x_0 x_1 x_0 x_2 + 4  x_0 x_1 x_0 x_3   +  3   
    x_0x_1^3 - \tfrac{9}{2}  x_0 x_1^2 x_2  -  x_0 x_1^2 x_3   -12  x_0 x_1 x_2 x_1 \\
&\ \ \ + 3 x_0 x_1 x_2^2 +  x_0 x_1 x_2 x_3 -6 x_0 x_1 x_3 x_1  -  x_0 x_1 x_3 x_2 
  + 8  x_0 x_2 x_0 x_2 \\ 
&\ \ \ + 2 x_0 x_2 x_0 x_3  + \tfrac{1}{2}x_0x_2x_1^2  - 2 x_0 x_2 x_1 x_2  -   x_0 x_2 x_1 x_3 
+  3  x_0 x_2^2 x_1   \\
&\ \ \  + x_0 x_2 x_3 x_1 - x_0 x_3 x_1^2 +  x_0 x_3 x_1 x_2  -  x_0 x_3 x_2 x_1 + \circlearrowright
\end{align*}
In particular, it is Calabi--Yau, completing the description of the components of $\QDef_0(\V)$.

\section{Quantization of actions of the affine group}
\label{sec:schwarz}

As explained in~\cite{Calvo-Andrade2007}, the $\E(3)$ Poisson structure arises from an action of the group $\G \cong \CC^\times \ltimes \CC$ of automorphisms of $\PP[1]$ that preserve $\infty \in \PP[1]$.  In this section, we briefly explain how this construction generalizes to other $\G$-varieties, and how this viewpoint allows us to produce some interesting bimodules over the quantizations.

Suppose that $\X$ is a projective variety carrying an action of $\G$, and that $\sL$ is an ample invertible sheaf on $\X$ that is equipped with a $\G$-equivariant structure.  For example, if $\X$ is Fano, then we can use the anticanonical line bundle.  The group $\G$ acts on the homogeneous coordinate ring
\[
\A(\X,\sL) = \bigoplus_{k \ge 0} \cohlgy[0]{\X,\sL^{\otimes k}}
\]
by automorphisms that preserve the grading.  Differentiating the $\G$-action, we obtain an action of the Lie algebra $\g = Lie(\G)$ by homogeneous derivations as in the previous section.
\begin{lemma}
The derivation $Y$ acts locally nilpotently on $\A(\X,\sL)$.
\end{lemma}

\begin{proof}
For $k \ge 0$, consider the action of $\g$ on $\W = \cohlgy[0]{\X,\sL^{\otimes k}}$.  We will show that $Y$ acts nilpotently on $\W$.  Indeed, $\W$ decomposes into generalized eigenspaces for the action of $X$, and the relation $[Y,X]=Y$ ensures that the action of $Y$ takes the generalized eigenspace with eigenvalue $\lambda$ to the generalized eigenspace with eigenvalue $\lambda - 1$.  Since $\W$ is finite-dimensional, there are only finitely many nonzero generalized eigenspaces, and hence $Y^n\W = 0$ for sufficiently large $n$.
\end{proof}

As a result, we can apply \autoref{lem:loc-nilp} and conclude that for any value of $\hbar \in \CC$, the CGG star product \eqref{eqn:cgg} defines an associate, noncommutative product on $\A(\X,\sL)$.  We denote the corresponding graded ring by $\A(\X,\sL,\hbar)$.

We see immediately that if $\Z \subset \X$ is a $\G$-equivariant subscheme, then the corresponding homogeneous ideal $\I \subset \A(\X,\sL)$ is also a two-sided ideal for $\A(\X,\sL,\hbar)$.  Hence the embedding $\Z \subset \X$ is quantized.

Similarly, if $\sE$ is a $\G$-equivariant sheaf on $\X$, we can use the CGG formula to make the graded vector space
\[
\M(\X,\sL,\sE) = \bigoplus_{k \in \ZZ} \cohlgy[0]{\X,\sL^{\otimes k}\otimes \sE}
\]
into a bimodule over $\A(\X,\sL,\hbar)$.  Thus, for every $\hbar \in \CC$, we have a functor which takes $\G$-equivariant sheaves on $\X$ to bimodules over $\A(\X,\sL,\hbar)$.  In particular, the sheaves of differential forms and vector fields are quantized.

\begin{example}
Let $\X = \PP[n]$ be the $n$th symmetric power of $\PP[1]$, on which $\G$ acts diagonally.  The ample line bundle $\sL = \sO{\PP[n]}(1)$, being a root of the anti-canonical bundle, carries a canonical action of $\G$.  We therefore obtain a quantization $\A(\PP[n],\sO{\PP[n]}(1),\hbar)$ of the polynomial ring in $n+1$ variables.

For $1 \le k \le n$, the embedding $\mu_k : \PP[1] \to \PP[n]$ that sends the point $p$ to the degree-$n$ divisor $(n-k)\cdot p + k\cdot\infty$ is $\G$-equivariant, and therefore quantizes to give commutative rational curves in the corresponding quantum $\PP[n]$.  When $n=2$, the curve $\mu_2(\PP[1])$ is a conic, and $\mu_1(\PP[1])$ is a line tangent to this conic, a situation familiar from the classifications in~\cite{Artin1990,Bondal1993a}.  

When $n=3$, we obtain the $\E(3)$ algebra of the previous section, which contains three commutative rational curves: a twisted cubic, a plane conic and a line that is tangent to both.  The two-dimensional $\G$-orbits are the symplectic leaves of the corresponding Poisson structure and form a pencil of singular sextic surfaces---the level sets of the $j$-invariant.  Hence we obtain noncommutative singular sextics, which is particularly interesting in light of the fact that there are no Poisson structures on a smooth sextic. \qed
\end{example}

\begin{example}
Let $\X = \PP[1] \times \PP[n-1]$ where $\PP[n-1]$ is viewed as the $(n-1)$st symmetric power of $\PP[1]$ and $\G$ acts diagonally on $\X$.  The divisor $\D = \{\infty\} \times \PP[n-1] \subset \X$ is $\G$-invariant and hence for every $k \in \ZZ$, the sheaf $\sO{\X}(k\D)$ has a canonical $\G$-equivariant structure.  For $k$ sufficiently large, its direct image under the symmetrization map $\pi : \PP[1]\times \PP[n-1] \to \PP[n]$ is an indecomposable vector bundle of rank $n$, known as a Schwarzenberger bundle~\cite{Schwarzenberger1961}.  By the above procedure, we can quantize the Schwarzenberger bundles to obtain graded bimodules for the algebra $\A(\PP[n],\sO{\PP[n]}(1),\hbar)$ constructed in the previous example.\qed
\end{example}

\begin{example}
In \cite{Loray2013}, one finds a classification of Poisson structures on Fano threefolds of Picard rank one.  Many of these Poisson structures are induced by an action of the group $\G$, and hence the procedure above can be used to quantize them.  Perhaps the most interesting example is the Mukai--Umemura threefold, which is embedded in the 12th symmetric power $\PP[12]$ of $\PP[1]$.  It is the closure of the $\Aut{{\PP[1]}}$-orbit of the divisor formed from the vertices of a regular icosahedron.\qed
\end{example}

\begin{example}
Let $\Gamma \subset \SL{2,\CC}$ be a finite subgroup and choose an embedding $\G \subset \SL{2,\CC}$.  Then any equivariant compactification of $\SL{2,\CC}/\Gamma$ inherits a $\G$-action.  $\SL{2,\CC}$-equivariant bundles on such compactifications---particularly $\PP[3]$ and the Mukai--Umemura threefold discussed above---play a central role in Hitchin's constructions~\cite{Hitchin1995a,Hitchin2010c} of algebraic solutions of the sixth Painlev\'e equation.   It would be interesting to understand how those solutions are linked with the corresponding quantum bimodules.\qed
\end{example}

\bibliographystyle{hyperamsplain}
\bibliography{p3-quant}

\providecommand{\bysame}{\leavevmode\hbox to3em{\hrulefill}\thinspace}
\providecommand{\MR}{\relax\ifhmode\unskip\space\fi MR }
\providecommand{\MRhref}[2]{%
  \href{http://www.ams.org/mathscinet-getitem?mr=#1}{#2}
}
\providecommand{\href}[2]{#2}
\begin{thebibliography}{10}

\bibitem{Artin1969}
M.~Artin, \emph{Algebraic approximation of structures over complete local
  rings}, Inst. Hautes \'Etudes Sci. Publ. Math. (1969), no.~36, 23--58.

\bibitem{Artin1990}
M.~Artin, J.~Tate, and M.~Van~den Bergh, \emph{Some algebras associated to
  automorphisms of elliptic curves}, The {G}rothendieck {F}estschrift, {V}ol.\
  {I}, Progr. Math., vol.~86, Birkh\"auser Boston, Boston, MA, 1990,
  pp.~33--85.

\bibitem{Artin1994}
M.~Artin and J.~J. Zhang, \emph{Noncommutative projective schemes},
  \href{http://dx.doi.org/10.1006/aima.1994.1087}{Adv. Math. \textbf{109}
  (1994)}, no.~2, 228--287.

\bibitem{Artin1987}
M.~Artin and W.~F. Schelter, \emph{Graded algebras of global dimension {$3$}},
  \href{http://dx.doi.org/10.1016/0001-8708(87)90034-X}{Adv. in Math.
  \textbf{66} (1987)}, no.~2, 171--216.

\bibitem{Bocklandt2010}
R.~Bocklandt, T.~Schedler, and M.~Wemyss, \emph{Superpotentials and higher
  order derivations}, \href{http://dx.doi.org/10.1016/j.jpaa.2009.07.013}{J.
  Pure Appl. Algebra \textbf{214} (2010)}, no.~9, 1501--1522.

\bibitem{Bondal1993a}
A.~I. Bondal and A.~E. Polishchuk, \emph{Homological properties of associative
  algebras: the method of helices},
  \href{http://dx.doi.org/10.1070/IM1994v042n02ABEH001536}{Izv. Ross. Akad.
  Nauk Ser. Mat. \textbf{57} (1993)}, no.~2, 3--50.

\bibitem{Bondal1993}
A.~I. Bondal, \emph{{Non-commutative deformations and Poisson brackets on
  projective spaces}}, Max-Planck-Institute Preprint (1993), no.~93-67.

\bibitem{Calvo-Andrade2007}
O.~Calvo-Andrade and F.~Cukierman, \emph{A note on the {$\jmath$} invariant and
  foliations}, Proceedings of the {XVI}th {L}atin {A}merican {A}lgebra
  {C}olloquium ({S}panish), Bibl. Rev. Mat. Iberoamericana, Rev. Mat.
  Iberoamericana, Madrid, 2007, pp.~99--108.

\bibitem{Campillo1999}
A.~Campillo and J.~Olivares, \emph{A plane foliation of degree different from 1
  is determined by its singular scheme},
  \href{http://dx.doi.org/10.1016/S0764-4442(99)80289-4}{C. R. Acad. Sci. Paris
  S\'er. I Math. \textbf{328} (1999)}, no.~10, 877--882.

\bibitem{Cassidy1999}
T.~Cassidy, \emph{Global dimension {$4$} extensions of {A}rtin-{S}chelter
  regular algebras}, \href{http://dx.doi.org/10.1006/jabr.1999.7902}{J. Algebra
  \textbf{220} (1999)}, no.~1, 225--254.

\bibitem{Cassidy2006}
T.~Cassidy, P.~Goetz, and B.~Shelton, \emph{Generalized {L}aurent polynomial
  rings as quantum projective 3-spaces},
  \href{http://dx.doi.org/10.1016/j.jalgebra.2005.10.027}{J. Algebra
  \textbf{303} (2006)}, no.~1, 358--372.

\bibitem{Cassidy2008}
T.~Cassidy and B.~Shelton, \emph{Generalizing the notion of {K}oszul algebra},
  \href{http://dx.doi.org/10.1007/s00209-007-0263-8}{Math. Z. \textbf{260}
  (2008)}, no.~1, 93--114.

\bibitem{Cerveau1996}
D.~Cerveau and A.~Lins~Neto, \emph{Irreducible components of the space of
  holomorphic foliations of degree two in {$\bold C{\rm P}(n)$}, {$n\geq 3$}},
  \href{http://dx.doi.org/10.2307/2118537}{Ann. of Math. (2) \textbf{143}
  (1996)}, no.~3, 577--612.

\bibitem{Coll1989}
V.~Coll, M.~Gerstenhaber, and A.~Giaquinto, \emph{An explicit deformation
  formula with noncommuting derivations}, Ring theory 1989 ({R}amat {G}an and
  {J}erusalem, 1988/1989), Israel Math. Conf. Proc., vol.~1, Weizmann,
  Jerusalem, 1989, pp.~396--403.

\bibitem{Dolgushev2009}
V.~Dolgushev, \emph{The {V}an den {B}ergh duality and the modular symmetry of a
  {P}oisson variety},
  \href{http://dx.doi.org/10.1007/s00029-008-0062-z}{Selecta Math. (N.S.)
  \textbf{14} (2009)}, no.~2, 199--228.

\bibitem{Drinfeld1992}
V.~G. Drinfel'd, \emph{On quadratic commutation relations in the quasiclassical
  case [translation of {\it {M}athematical physics, functional analysis
  ({R}ussian)}, 25--34, 143, ``{N}aukova {D}umka'', {K}iev, 1986; {MR}0906075
  (89c:58048)]}, Selecta Math. Soviet. \textbf{11} (1992), no.~4, 317--326.
  Selected translations.

\bibitem{Dubois-Violette2005}
M.~Dubois-Violette, \emph{Graded algebras and multilinear forms},
  \href{http://dx.doi.org/10.1016/j.crma.2005.10.017}{C. R. Math. Acad. Sci.
  Paris \textbf{341} (2005)}, no.~12, 719--724.

\bibitem{Dubois-Violette2007}
\bysame, \emph{Multilinear forms and graded algebras},
  \href{http://dx.doi.org/10.1016/j.jalgebra.2007.02.007}{J. Algebra
  \textbf{317} (2007)}, no.~1, 198--225.

\bibitem{Etingof2010}
P.~Etingof and V.~Ginzburg, \emph{Noncommutative del {P}ezzo surfaces and
  {C}alabi-{Y}au algebras}, \href{http://dx.doi.org/10.4171/JEMS/235}{J. Eur.
  Math. Soc. (JEMS) \textbf{12} (2010)}, no.~6, 1371--1416,
  \href{http://arxiv.org/abs/0709.3593}{{\tt 0709.3593}}.

\bibitem{GAP4}
The GAP~Group, \emph{{GAP -- Groups, Algorithms, and Programming, Version
  4.6.5}}, 2013.

\bibitem{Ginzburg2006}
V.~Ginzburg, \emph{{C}alabi-{Y}au algebras},
  \href{http://arxiv.org/abs/math/0612139}{{\tt math/0612139}}.

\bibitem{Goetz2003}
P.~D. Goetz, \emph{The noncommutative algebraic geometry of quantum projective
  spaces}, ProQuest LLC, Ann Arbor, MI, 2003. Thesis (Ph.D.)--University of
  Oregon.

\bibitem{Goodman2013}
J.~Goodman and U.~Kraehmer, \emph{Untwisting a twisted {C}alabi-{Y}au algebra},
  \href{http://arxiv.org/abs/1304.0749}{{\tt 1304.0749}}.

\bibitem{He2013}
J.-W. He, F.~Van~Oystaeyen, and Y.~Zhang, \emph{Calabi-{Y}au algebras and their
  deformations}, Bull. Math. Soc. Sci. Math. Roumanie (N.S.) \textbf{56(104)}
  (2013), no.~3, 335--347.

\bibitem{Hitchin1995a}
N.~J. Hitchin, \emph{Poncelet polygons and the {P}ainlev\'e equations},
  Geometry and analysis ({B}ombay, 1992), Tata Inst. Fund. Res., Bombay, 1995,
  pp.~151--185.

\bibitem{Hitchin2010c}
N.~Hitchin, \href{http://dx.doi.org/10.1090/conm/522/10292}{\emph{Vector
  bundles and the icosahedron}}, Vector bundles and complex geometry, Contemp.
  Math., vol. 522, Amer. Math. Soc., Providence, RI, 2010, pp.~71--87.

\bibitem{Holzer2008}
S.~Holzer and O.~Labs, \emph{{\sc surfex 0.90}}, Tech. report, University of
  Mainz, University of Saarbr\"ucken, 2008, {\tt
  www.surfex.AlgebraicSurface.net}.

\bibitem{Kontsevich2001}
M.~Kontsevich, \emph{Deformation quantization of algebraic varieties},
  \href{http://dx.doi.org/10.1023/A:1017957408559}{Lett. Math. Phys.
  \textbf{56} (2001)}, no.~3, 271--294. EuroConf{\'e}rence Mosh{\'e} Flato
  2000, Part III (Dijon).

\bibitem{Kontsevich2003}
\bysame, \emph{Deformation quantization of {P}oisson manifolds},
  \href{http://dx.doi.org/10.1023/B:MATH.0000027508.00421.bf}{Lett. Math. Phys.
  \textbf{66} (2003)}, no.~3, 157--216.

\bibitem{LeBruyn1993}
L.~Le~Bruyn and S.~P. Smith, \emph{Homogenized {$\mathfrak{sl}(2)$}},
  \href{http://dx.doi.org/10.2307/2160112}{Proc. Amer. Math. Soc. \textbf{118}
  (1993)}, no.~3, 725--730.

\bibitem{LeBruyn1996}
L.~Le~Bruyn, S.~P. Smith, and M.~Van~den Bergh, \emph{Central extensions of
  three-dimensional {A}rtin-{S}chelter regular algebras},
  \href{http://dx.doi.org/10.1007/PL00004532}{Math. Z. \textbf{222} (1996)},
  no.~2, 171--212.

\bibitem{Levasseur1993}
T.~Levasseur and S.~P. Smith, \emph{Modules over the {$4$}-dimensional
  {S}klyanin algebra}, Bull. Soc. Math. France \textbf{121} (1993), no.~1,
  35--90.

\bibitem{Liu1992}
Z.-J. Liu and P.~Xu, \emph{On quadratic Poisson structures},
  \href{http://dx.doi.org/10.1007/BF00420516}{Lett. Math. Phys. \textbf{26}
  (1992)}, no.~1, 33--42.

\bibitem{Loray2013}
F.~Loray, J.~V. Pereira, and F.~Touzet, \emph{Foliations with trivial canonical
  bundle on {F}ano 3-folds},
  \href{http://dx.doi.org/10.1002/mana.201100354}{Math. Nachr. \textbf{286}
  (2013)}, no.~8-9, 921--940.

\bibitem{Lu2007}
D.-M. Lu, J.~H. Palmieri, Q.-S. Wu, and J.~J. Zhang, \emph{Regular algebras of
  dimension 4 and their {$A_\infty$}-{E}xt-algebras},
  \href{http://dx.doi.org/10.1215/S0012-7094-07-13734-7}{Duke Math. J.
  \textbf{137} (2007)}, no.~3, 537--584.

\bibitem{Moyal1949}
J.~E. Moyal, \emph{Quantum mechanics as a statistical theory}, Proc. Cambridge
  Philos. Soc. \textbf{45} (1949), 99--124.

\bibitem{Polishchuk1997}
A.~Polishchuk, \emph{Algebraic geometry of {P}oisson brackets}, J. Math. Sci.
  (N. Y.) \textbf{84} (1997), no.~5, 1413--1444. Algebraic geometry, 7.

\bibitem{Polishchuk2005}
A.~Polishchuk and L.~Positselski, \emph{Quadratic algebras}, University Lecture
  Series, vol.~37, American Mathematical Society, Providence, RI, 2005.

\bibitem{Pym2013}
B.~Pym, \emph{{Poisson Structures and Lie Algebroids in Complex Geometry}},
  Ph.D. thesis, University of Toronto, 2013.

\bibitem{Reyes2013}
M.~Reyes, D.~Rogalski, and J.~J. Zhang, \emph{Skew {C}alabi-Yau Algebras and
  Homological Identities}, \href{http://arxiv.org/abs/1302.0437}{{\tt
  1302.0437}}.

\bibitem{Schwarzenberger1961}
R.~L.~E. Schwarzenberger, \emph{Vector bundles on the projective plane}, Proc.
  London Math. Soc. (3) \textbf{11} (1961), 623--640.

\bibitem{Shoikhet2010}
B.~Shoikhet, \emph{Koszul duality in deformation quantization and {T}amarkin's
  approach to {K}ontsevich formality},
  \href{http://dx.doi.org/10.1016/j.aim.2009.12.010}{Adv. Math. \textbf{224}
  (2010)}, no.~3, 731--771.

\bibitem{Sklyanin1982}
E.~K. Sklyanin, \emph{Some algebraic structures connected with the
  {Y}ang-{B}axter equation}, Funktsional. Anal. i Prilozhen. \textbf{16}
  (1982), no.~4, 27--34, 96.

\bibitem{Smith1992}
S.~P. Smith and J.~T. Stafford, \emph{Regularity of the four-dimensional
  {S}klyanin algebra}, Compositio Math. \textbf{83} (1992), no.~3, 259--289.

\bibitem{Stafford2001}
J.~T. Stafford and M.~van~den Bergh, \emph{Noncommutative curves and
  noncommutative surfaces},
  \href{http://dx.doi.org/10.1090/S0273-0979-01-00894-1}{Bull. Amer. Math. Soc.
  (N.S.) \textbf{38} (2001)}, no.~2, 171--216.

\bibitem{Staniszkis1996}
J.~M. Staniszkis, \emph{Linear modules over {S}klyanin algebras}, J. London
  Math. Soc. (2) \textbf{53} (1996), no.~3, 464--478.

\bibitem{Sage}
W.~Stein et~al., \emph{{S}age {M}athematics {S}oftware ({V}ersion 5.4)}, The
  Sage Development Team, 2013. {\tt http://www.sagemath.org}.

\bibitem{Tate1996}
J.~Tate and M.~van~den Bergh, \emph{Homological properties of {S}klyanin
  algebras}, \href{http://dx.doi.org/10.1007/s002220050065}{Invent. Math.
  \textbf{124} (1996)}, no.~1-3, 619--647.

\bibitem{Vey1975}
J.~Vey, \emph{D\'eformation du crochet de {P}oisson sur une vari\'et\'e
  symplectique}, Comment. Math. Helv. \textbf{50} (1975), no.~4, 421--454.

\bibitem{Weinstein1997}
A.~Weinstein, \emph{The modular automorphism group of a {P}oisson manifold},
  \href{http://dx.doi.org/10.1016/S0393-0440(97)80011-3}{J. Geom. Phys.
  \textbf{23} (1997)}, no.~3-4, 379--394.

\bibitem{Xu1999}
P.~Xu, \emph{Gerstenhaber algebras and {BV}-algebras in {P}oisson geometry},
  \href{http://dx.doi.org/10.1007/s002200050540}{Comm. Math. Phys. \textbf{200}
  (1999)}, no.~3, 545--560.

\bibitem{Zhang1996}
J.~J. Zhang, \emph{Twisted graded algebras and equivalences of graded
  categories}, \href{http://dx.doi.org/10.1112/plms/s3-72.2.281}{Proc. London
  Math. Soc. (3) \textbf{72} (1996)}, no.~2, 281--311.

\bibitem{Zhang2009}
J.~J. Zhang and J.~Zhang, \emph{Double extension regular algebras of type
  (14641)}, \href{http://dx.doi.org/10.1016/j.jalgebra.2009.03.041}{J. Algebra
  \textbf{322} (2009)}, no.~2, 373--409.

\end{thebibliography}
\end{document}